\documentclass[12pt]{amsart}
\usepackage[pdftex]{hyperref}
\hypersetup{colorlinks,citecolor=blue,linktocpage,hyperindex=true,backref=true}
\usepackage{amsmath,amscd,amssymb,euscript}
\usepackage{mathtools}
\usepackage{latexsym}
\usepackage{graphicx}
\usepackage{tikz}
\usetikzlibrary{backgrounds}
\usetikzlibrary{decorations.fractals}
\usetikzlibrary{calc,intersections,through,backgrounds}
\usetikzlibrary{arrows.meta}
\tikzset{In/.tip = {Hooks[right]}}
\tikzset{Onto/.tip = {To[sep] To}}
\tikzset{Eq/.style = {-, double equal sign distance}}
\tikzset{Iso/.style = {edge node = {node[above] {$\sim$}}, inner sep = 0}}
\usetikzlibrary{positioning}
\usepackage{mathrsfs}
\usepackage{epstopdf}
\usepackage{rotating}
\usepackage{lscape}
\input xy
\xyoption{all}

\newcommand{\bC}{{\mathbb C}}

\newcommand{\bK}{{\mathbb K}}
\newcommand{\bL}{{\mathbb L}}
\newcommand{\bN}{{\mathbb N}}
\newcommand{\bP}{{\mathbb P}}
\newcommand{\bQ}{{\mathbb Q}}

\newcommand{\bW}{{\mathbb W}}
\newcommand{\bZ}{{\mathbb Z}}
\newcommand{\cA}{{\mathcal A}}

\newcommand{\cC}{{\mathcal C}}
\newcommand{\cD}{{\mathcal D}}

\newcommand{\cG}{{\mathcal G}}
\newcommand{\cH}{{\mathcal H}}

\newcommand{\cM}{{\mathcal M}}

\newcommand{\cO}{{\mathcal O}}
\newcommand{\cP}{{\mathcal P}}

\newcommand{\cR}{{\mathcal R}}
\newcommand{\cS}{{\mathcal S}}

\newcommand{\cU}{{\mathcal U}}
\newcommand{\cV}{{\mathcal V}}

\newcommand{\ocM}{{\overline{\mathcal M}}}

\newcommand{\oW}{{\overline{W}}}

\newcommand{\tB}{{\widetilde{B}}}
\newcommand{\tC}{{\widetilde{C}}}

\newcommand{\tcC}{{\widetilde{\mathcal C}}}

\newcommand{\tlambda}{{\widetilde{\lambda}}}
\newcommand{\tsigma}{{\widetilde{\sigma}}}
\newcommand{\tmu}{{\widetilde{\mu}}}
\newcommand{\tzeta}{{\widetilde{\zeta}}}

\newcommand{\tS}{{\widetilde{S}}}

\newcommand{\trho}{{\widetilde{\rho}}}

\newcommand{\tU}{{\widetilde{U}}}

\newcommand{\tV}{{\widetilde{V}}}
\newcommand{\tX}{{\widetilde{X}}}
\newcommand{\tY}{{\widetilde{Y}}}
\newcommand{\tZ}{{\widetilde{Z}}}

\newcommand{\ra}{\rightarrow}
\newcommand{\lra}{\longrightarrow}

\newcommand{\inj}{\hookrightarrow}

\newcommand{\T}{\Theta}

\newcommand{\aut}{\operatorname{Aut}}
\newcommand{\Bl}{\operatorname{Bl}}

\newcommand{\Img}{\operatorname{Im}}
\newcommand{\Ker}{\operatorname{Ker}}

\newcommand{\Pic}{\operatorname{Pic}}

\newcommand{\iffw}{if and only if}

\newcommand{\pt}{\zeta}

\usepackage[nameinlink,capitalize]{cleveref}
\usepackage{enumitem}
\newcommand{\ol}[1]{\overline{#1}}

\newcommand\aligneq{{\phantom{{}={}}}}
\newcommand\eulerian[2]{\genfrac{\langle}{\rangle}{0ex}{}{#1}{#2}}
\newcommand\nowidth[1]{\makebox[0em][l]{#1}}
\newlist{anumerate}{enumerate}{1}
\setlist[anumerate]{label=(\alph*)}
\crefname{subsection}{Paragraph}{Paragraphs}
\Crefname{subsection}{Paragraph}{Paragraphs}
\crefname{equation}{}{}
\Crefname{equation}{}{}
\crefname{anumeratei}{}{}
\Crefname{anumeratei}{}{}
\newcommand\dual[1]{{#1}^\vee}
\newcommand\eq\coloneqq
\newcommand\into\hookrightarrow

\newcommand\iso\cong
\newcommand\minus\setminus
\newcommand\onto\twoheadrightarrow
\newcommand\prym[2]{\cP_{#1}(#2)}
\newcommand\oddprym[2]{\cP_{#1}^-(#2)}
\newcommand\bnprym[3]{\cP_{#2}^{#1}(#3)}
\newcommand\prymap{\cP_0}
\newcommand\Tprym[1]{\T_{#1}}
\newcommand\pr{{\mathrm{pr}}}

\newcommand\tensor\otimes
\newcommand\tet{\overset{\triangle}{\sim}}
\renewcommand\tilde\widetilde
\newcommand{\oL}{{\overline{L}}}

\newcommand{\tR}{{\widetilde{R}}}
\newcommand{\tT}{{\widetilde{T}}}

\newcommand{\teta}{{\widetilde\eta}}
\newcommand{\tiota}{{\widetilde\iota}}
\newcommand{\ttheta}{{\widetilde\theta}}
\newcommand{\txi}{{\widetilde\xi}}

\newcommand{\esix}{\mathsf{E}_6}
\newcommand{\Tt}{{\widetilde\Theta}}
\newcommand{\tW}{{\widetilde{W}}}
\newcommand{\tXi}{{\widetilde\Xi}}
\newcommand\tou[1]{\xrightarrow{\,{#1}\,}}

\DeclareMathOperator\ch{ch}
\DeclareMathOperator\chow{CH}
\DeclareMathOperator\alg{alg}
\newcommand\achow{\chow^{\alg}}
\DeclareMathOperator\id{id}
\DeclareMathOperator\nm{Nm}
\DeclareMathOperator\todd{td}
\DeclarePairedDelimiter\abs{|}{|}
\DeclarePairedDelimiter\card{|}{|}
\DeclarePairedDelimiter\pair{\langle}{\rangle}
\DeclarePairedDelimiter\resno{}{|}
\DeclarePairedDelimiter\set{\{}{\}}
\newcommand\res[2]{\resno{#1}_{#2}}

\renewcommand\emptyset\varnothing
\newenvironment{tikzcenter}[1][node distance = 10ex and 7em]{%
	\begin{center}%
		\begin{tikzpicture}[auto, on grid, #1]
}{
		\end{tikzpicture}%
	\end{center}%
}

\theoremstyle{definition}

\newtheorem{proposition}{Proposition}[section]
\newtheorem{lemma}[proposition]{Lemma}

\newtheorem{theorem}[proposition]{Theorem}
\newtheorem{theoremi}{Theorem}

\newtheorem{definition}[proposition]{Definition}
\newtheorem{corollary}[proposition]{Corollary}

\newtheorem{remark}[proposition]{Remark}

\numberwithin{equation}{section}

\begin{document}

\title[Surfaces generating the primal cohomology]{Surfaces generating the even primal cohomology of an abelian fivefold}

\author{Jonathan Conder}

\address{Department of Mathematics, University of California San Diego, 9500 Gilman Drive \# 0112, La Jolla, CA 92093-0112, USA}

\email{jconder@math.ucsd.edu, jonno.conder@gmail.com}

\author{Edward Dewey}

\address{Department of Mathematics, University of California San Diego, 9500 Gilman Drive \# 0112, La Jolla, CA 92093-0112, USA}

\email{ehdewey@math.ucsd.edu, ed.dewey@gmail.com}

\author{Elham Izadi}

\address{Department of Mathematics, University of California San Diego, 9500 Gilman Drive \# 0112, La Jolla, CA 92093-0112, USA}

\email{eizadi@math.ucsd.edu}

\thanks{}


\begin{abstract}

Given a very general abelian fivefold $A$ and a principal polarization $\Theta \subset A$, we construct surfaces generating the algebraic part of the middle cohomology $H^4(\Theta, {\mathbb Q})$, and determine the intersection pairing between these surfaces. In particular, we obtain a new proof of the Hodge conjecture for $H^4(\Theta, {\mathbb Q})$ and show that it contains a copy of the root lattice of $E_6$.

\end{abstract}

\maketitle

\tableofcontents

\section*{Introduction}

Let $A$ be a principally polarized abelian variety (ppav) of dimension $g \ge 4$ with smooth
symmetric theta divisor $\T$.
By the Lefschetz hyperplane theorem and Poincar\'e Duality (see, e.g., \cite{IzadiWang2015}), the cohomology of $\T$ is determined by that of $A$ except in the middle dimension $g - 1$.
The primitive cohomology of $\T$, in the sense of Lefschetz, is
\[
	H_\pr^{g - 1}(\T, \bZ) \eq \Ker\left(H^{g - 1}(\T, \bZ) \tou{\cup [\T]} H^{g + 1}(\T, \bZ)\right).
\]
The primal cohomology of $\T$ is defined as (see \cite{IzadiWang2015} and \cite{IzadiTamasWang})
\[
	\bK \eq \Ker\left(H^{g - 1}(\T, \bZ) \tou{i_*} H^{g + 1}(A, \bZ)\right)
\] 
where $i : \T \into A$ is the inclusion.
This is a Hodge substructure of $H_\pr^{g - 1}(\T, \bZ)$ of rank $g! - \frac{1}{g + 1} {2 g \choose g}$ and level $g - 3$ while the primitive cohomology $H_\pr^{g - 1}(\T, \bZ)$ has full level $g - 1$.

The action of $-1$ splits $\bK_\bQ$ into the direct sum of its invariant piece $\bK_\bQ^+$ and its anti-invariant piece $\bK_\bQ^-$. As shown in \cite[Lemma 6.1]{IzadiWang2018}, the Hodge structure $\bK^{(-1)^g}$ has level $g-3$ while the Hodge structure $\bK^{(-1)^{g-1}}$ has level $g-5$.

The primal cohomology $\bK$ and its Hodge substructure $\bK^{(-1)^{g-1}}$ are therefore interesting test cases for the general Hodge conjecture.
The general Hodge conjecture predicts that $\bK_\bQ \eq \bK \tensor \bQ$ is contained in the image, via Gysin pushforward, of the cohomology of a smooth (possibly reducible) variety of pure dimension $g - 3$ (see \cite{IzadiWang2015}).
This conjecture was proved in \cite{IzadiVanStraten} and \cite{IzadiTamasWang} in the cases $g = 4$ and $g = 5$.
When $g = 4$, it also follows from the proof of the Hodge conjecture in \cite{IzadiVanStraten} that, for $(A, \T)$ generic, $\bK$ is a simple Hodge structure (isogenous to the third cohomology of a smooth cubic threefold).
In this case the primal cohomology is fixed under the action of $-1$.

In the case $g = 5$, $\bK_\bQ^+$ and $\bK_\bQ^-$ have respective dimensions $6$ and $72$. The space $\bK_\bQ^+$ consists of Hodge classes
(\cite[Corollary~6.2]{IzadiWang2018}) while $\bK_\bQ^-$ is simple \cite{IzadiWang2018}.
It follows from the main result of \cite{IzadiTamasWang} and the Lefschetz $(1, 1)$ theorem that the classes belonging to $\bK_\bQ^+$ are algebraic.
Here we describe explicit natural surfaces in $\T$ which represent these classes.
Our main result is the following
\begin{theoremi}\label{thmmain}
Suppose $(A, \T)$ is a general ppav of dimension $5$.
\begin{anumerate}
\item
	There are $27$ smooth surfaces $V_i \subset \T$ (to be described below, up to translation), whose classes in $H^4(\T, \bQ)$ span $\bQ [\T]^2 + \bK_\bQ^+$ (which is the space of Hodge classes if $A$ is very general).
\item
	The sublattice of $\bK^+$ spanned by classes of the form $[V_i] - [V_j]$ is isometric to
	$H_\pr^2(X, \bZ)(-2)$ for any smooth cubic surface $X \subset \bP^3$.
\item
	There is a (non-canonical) bijection between the $V_i$ and the lines $L_i$ in $X$ such that the isometry sends $[V_i] - [V_j]$ to $[L_i] - [L_j]$.
\end{anumerate}
\end{theoremi}

We construct the surfaces in two different ways, both of which rely on the theory of Prym varieties.
One uses Brill-Noether theory for Prym varieties \cite{Welters85}; the other exhibits the surfaces as special subvarieties of $A$, in the sense of \cite{Beauville82}.
By comparing these two constructions, exploiting a connection with the $27$ lines on a cubic surface, we produce relations between the intersection numbers $[V_i] . [V_j]$.
With these in hand, it remains to compute $[V_i]^2$ for all $i$.

To do so, we adapt the one parameter degeneration of \cite{IzadiWang2018}, whose central fiber is a compactification $(A_0, \T_0)$ of a semiabelian extension of a ppav $(B, \Xi)$ of dimension $4$.
The limit theta divisor $\T_0$ is singular and birational to $B$.

Each limit surface $V_0 \subset \T_0$ is birational to $\Xi_\alpha \cap \Xi_\beta$ for some $\alpha, \beta \in B$ (subscripts denote translation).
We identify $V_0$ using a Hilbert polynomial calculation, and make sense of $[V_0]^2$ using the smoothness of the total space of theta divisors.
Finally, we compute the degree of $[V_0]^2$ using properties of the Prym-embedded curves in $\Xi_\alpha \cap \Xi_\beta$, which were studied in \cite{Izadi-A4-1995} and \cite{Kramer2015}.

The paper is organized as follows. In \cref{secPrym} we gather some known facts about Prym Varieties and prove some preliminary results about some subvarieties of Prym varieties that we will need later. In \cref{secSurfaces} we define the $27$ surfaces and compute their intersection numbers with the exception of their self-intersection numbers. \cref{secprymcurves} contains analogous results for sets of $27$ curve classes in abelian fourfolds that we will need later for our degeneration argument. In \cref{secfamily} we describe the one-parameter degeneration of abelian fivefolds and their theta divisors that we will use to prove our main result. In \cref{the family of surfaces} we describe the one-parameter degenerations of surfaces that we use to compute the self-intersection numbers of the surfaces $V_i$, and in \cref{nodal fiber} we describe the central fibers of these families. We complete the proof of our main result in \cref{secProof} by computing the self-intersection numbers of $V_i$, using the one-parameter families of surfaces. Finally, the Appendix, \cref{secApp}, contains some computations used in the paper. It also contains the computation of the ranks and Hodge numbers of $\bK^+$ and $\bK^-$ in all dimensions.

\section*{Notation}

If $C$ is a curve, then $\tC$ always denotes a fixed \'etale double cover of $C$, with covering
involution $\sigma : \tC \to \tC$ and quotient map $\pi : \tC \to C$.
For $c \in \tC$ we denote $\sigma(c)$ by $c'$.
We denote by $C^{(d)}$ the
$d^{\text{th}}$ symmetric power of $C$, and $G_d^r(C)$ (resp.\ $W_d^r(C)$) the space of linear
systems (resp.\ complete linear systems) of degree $d$
and dimension $r$ (resp.\ at least $r$) on $C$.
We typically denote a fixed element of $G_d^r(C)$ by $g_d^r$.
The dual projective space of $g_d^r$ is denoted by $\dual{g_d^r}$.
The notation $C_1 \tet C_2$ means $C_1$ and $C_2$ are tetragonally related (see \cref{tetragonal}).

As usual, $\cA_g$ is the (coarse) moduli space of dimension $g$ principally polarized abelian
varieties (ppav), $\cR_g$ is the (coarse) moduli space of \'{e}tale couble covers $\tC \to C$ with $C$ of genus
$g$,  $\cM_g$ is the (coarse) moduli space of smooth curves of genus $g$, and $\ol{\cM}_g$ is its
compactification parametrizing stable curves.

If $X$ is a scheme, then $\chow_d(X)$ is the Chow group of algebraic cycles on $X$ modulo rational equivalence.

If $L$ is a lattice then $L(n)$ is the lattice obtained by multiplying the intersection form by $n \in \bZ$.

\section{Prym varieties and Prym-embeddings}\label{secPrym}

\subsection{Prym varieties}
The \emph{Prym map} $\prymap : \cR_{g + 1} \to \cA_g$ sends a Beauville admissible cover $\tX \tou\pi X$ (see \cite{Beauville1977-1}) to its Prym variety $\prym{0}{X} \eq \Img(\sigma^* - \id) \subset \Pic^0(\tX)$, where $\sigma : \tX \to \tX$ is the covering involution.
The map $\prymap$ is surjective for $g \le 5$, hence generically finite for $g = 5$ by a dimension count.
Its degree is $27$ in this case \cite{DonagiSmith81}.

\subsection{Prym torsors}
The Prym variety $\prym{0}{X}$ can also be defined as the identity component of $\Ker(\nm)$, where $\nm : \Pic(\tX) \to \Pic(X)$ is the norm map (which agrees with the push-forward on Chow groups $\chow_0(\tX) \tou{\pi_*} \chow_0(X)$).
The kernel of $\nm$ has a second component, a translate of $\prym{0}{X}$, which we denote by $\oddprym{0}{X}$.
We will also use the fiber of $\nm$ over $\omega_X$, whose components $\prym{2 g}{X}$ and $\oddprym{2 g}{X}$ consist of line bundles $L$ such that $h^0(L)$ is even or odd respectively.

\subsection{Prym-embeddings}
If $X$ is not hyperelliptic, then neither is $\tX$, as the push-forward $\pi_* : \tX^{(2)} \to X^{(2)}$ preserves linear equivalence.
In this case $\iota(p) \eq \cO_\tX(p - p')$ defines an embedding $\tX \into \oddprym{0}{X}$ (recall that $p' \eq \sigma(p)$).
If $Z \subseteq \Pic(\tX)$, then translates of $\iota (\tX)$ contained in $Z$ are called
\emph{Prym-embeddings} of $\tX$ in $Z$.

We also have the canonical morphism
\[
\begin{array}{lcl}
\tX^{(2)} & \lra & \prym{0}{X}\\
p+q & \longmapsto & [p, q] \eq \iota(p) + \iota(q) = \cO_\tX(p + q - p' - q').
\end{array}
\]
Note that $[p, q] = \kappa(p) - \kappa(q')$ for any Prym-embedding $\kappa : \tX \inj \prym{0}{X}$ (in fact any translate of $\iota$).

\subsection{Brill-Noether loci}\label{lambda spaces}
Given $r > 0$, let $\bnprym{r}{2g}{X}$ be the locus in $\prym{2g}{X} \amalg \oddprym{2g}{X}$ where $h^0 > r$ and $h^0 \not\equiv r \pmod{2}$.
It inherits a scheme structure from the classical Brill-Noether locus $W_{2g}^r(\tX) \subseteq \Pic(\tX)$,
and is smooth at all $L$ with $h^0 (L)=r+1$ if $X \in \cM_{g + 1}$ is sufficiently general \cite[Propsition~1.9 and Theorem~1.11]{Welters85}.

The effective locus $\Tprym{2g} = \Tprym{2g}(X) \eq \bnprym{1}{2g}{X} \subset \prym{2g}{X}$ defines a theta divisor in the sense that, for each $L_0 \in \prym{2g}{X}$, the translate $\Tprym{2g} - L_0$ is a theta divisor for $\prym{0}{X}$.
Choosing $L_0$ determines a group structure on $\prym{2g}{X}$, where inversion is given by $L \mapsto L_0^2 - L$.
If $L_0$ is a theta characteristic (meaning $L_0^2 \iso \omega_\tX$), this morphism (the \emph{residuation} map) preserves $h^0(L)$ by Serre duality.
In this case $\Tprym{2g} - L_0$ is a symmetric theta divisor.

By \cite[Proposition~3.11]{Izadi-A4-1995}, the second Brill-Noether locus $\tX_\lambda \eq \bnprym{2}{2g}{X} \subseteq \oddprym{2g}{X}$ consists of those $L$ for which $\iota(\tX) + L \subset \Tprym{2g}$.
Each $p \in \tX$ defines an embedding
\[
	W_p \eq \tX_\lambda + \iota(p) = \set{L \in \prym{2 g}{X} \mid h^0(L(-p)) > 1} \subset \Tprym{2g}.
\]
When $g = 5$, these are the surfaces we will use to generate $\bK^+$.
In order to understand them better, we will show that the $W_p$ are special subvarieties in the sense of Beauville \cite{Beauville82}.

\subsection{Special subvarieties}
Given a $g_d^r$ on $X$ with $2 r < d \leq 2 g$, the associated \emph{special subvarieties} of the symmetric power $\tX^{(d)}$ are the connected components $S_i$ of the fiber product
\begin{tikzcenter}[node distance = 10ex and 7em]
	\node (S) {$S = S_1 \amalg S_2$};
	\node[right = of S] (tXd) {$\tX^{(d)}$};
	\node[below = of S] (grd) {$\bP^r = g_d^r$};
	\node[below = of tXd] (Xd) {$X^{(d)}$\nowidth{.}};

	\draw[In-To]
		(S) edge (tXd)
		(grd) edge (Xd);
	
	\draw[-To]
		(S) edge (grd)
		(tXd) edge (Xd);
\end{tikzcenter}
If the base locus of the $g_d^r$ is reduced, then so is $S$.
When $d \le r + g$, we also consider the special subvarieties $T_i \subseteq \tX^{(2 g - d)}$ associated to the residual linear system $g_{2 g - d}^{r + g - d}$.
After choosing the indices appropriately, for each $D \in T_i$, the image of $S_i + D \subseteq \tX^{(2 g)}$ in $\Pic^{2g}(\tX)$ is a subvariety $V_D \subseteq \Tprym{2g}$.
On the other hand $S_{3 - i} + D$ maps into $\oddprym{2g}{X}$.
The $V_D$ are called \emph{special subvarieties} of $\Tprym{2g}$.

\subsection{The tetragonal construction}\label{tetragonal}
For a (base point free) $g_4^1$ the above is known as the tetragonal construction \cite{donagi92}.
Each special subvariety $\tX_i \eq S_i$ is a smooth curve, assuming the fibers of $X \to \dual{g_4^1}\cong \bP^1$ have at most one ramification point, with index at most $3$.
The quotients $\tX_i \to X_i$ induced by the covering involution $\tX^{(4)} \tou{\sigma_*}
\tX^{(4)}$ 
have the same Prym variety as $\tX \to X$.
Each $X_i$ carries a $g_4^1$ for which the associated special subvarieties of $\tX_i^{(4)}$ are $\tX$ and $\tX_{3 - i}$.
We say that the $\tX_i \to X_i$ are \emph{tetragonally related} to $\tX \to X$, written $X_i \tet
X$, or that $(X, X_1, X_2)$ is a \emph{tetragonal triple}.

If $X \tet Y$, given a Prym-embedding $\kappa : \tX \into \prym{0}{X}$ and a lift $D\in \tY^{(2 g - 4)}$ of a divisor in the $g_{2 g - 4}^{g - 3}$ on $Y$, there is a line bundle $L \in \prym{2 g}{Y}$ and an isomorphism $\varphi : \prym{0}{X} \ra \prym{0}{Y}$ such that $\varphi (\kappa(p)) = L^{-1}(P + D)$ for all $p \in \tX$ corresponding to $P \in \tY^{(4)}$.
In particular $\varphi(\kappa(\tX)) = V_D - L$.
Moreover, if $P, Q' \in \tY^{(4)}$ correspond to $p, q' \in \tX$, then
\[
	\varphi ([p, q]) = L^{-1}(P + D) \tensor L(-Q' - D) = \cO_\tY(P - Q').
\]

\subsection{Prym-curves}
A \emph{Prym-curve} for an abelian variety $A$ is an admissible cover $\tX \tou\pi X$ such that $\prym{0}{X} \iso A$.
By abuse of notation, we will often denote a Prym-curve by its base curve $X$, the double cover $\tX \tou\pi X$ being implicit.

\begin{definition}
We call a Prym-curve $\tX \tou\pi X$ \emph{good} if $X$ is smooth (hence $\pi$ is \'etale) and $X$ is not hyperelliptic, trigonal or bielliptic.
\end{definition}
When $A$ is general and $g\geq 5$, every Prym-curve for $A$ is good (see \cite[{\S}7]{Mumford-PrymI-1974}, \cite{recillas74} and \cite[{\S}3]{donagi92}). When $A$ is general and $g=4$, the fiber of the Prym map at $A$ always contains singular Prym curves. However, every smooth Prym-curve is good.

\begin{lemma}\label{manyIntersections}
	Let $X$ be a good Prym-curve and set $\T \eq \Tprym{2g}(X)$.
	If $p, q, r \in \tX$ are such that $p$, $p'$, $q$ and $r$ are distinct, then:
	\begin{anumerate}
	\item\label{doubleIntersection}
		$\T \cap \T_{[p, q]} = V_{p + q}$.
	\item\label{tripleIntersection}
		$\T \cap \T_{[p, q]} \cap \T_{[p, r]} = W_p \cup V_{p + q + r}$.
	\item\label{shiftedIntersection}
		$\T_{[p, q]} \cap \T_{[p, r]} \cap \T_{[q, r]} = (W_p + [q, r]) \cup V_{p + q + r}$.
	\item\label{algequivalent}
		$W_p$ and $W_p + [q, r]$ are algebraically equivalent in $\T_{[p, q]} \cap \T_{[p, r]}$.
	\end{anumerate}
	If $s \in \tX$ is such that $\pi_*(p + q + r + s) \in X^{(4)}$ moves in a pencil, then
	\begin{anumerate}[resume]
	\item\label{subordinateCase}
		$V_{p + q + r} = V_{p + q + r + s} \cup V_{p + q + r + s'}$.
	\end{anumerate}
\end{lemma}
\begin{proof}
	See the proof of \cite[Proposition~1]{BeauvilleDebarre87} for \cref{doubleIntersection,tripleIntersection}, and \cite[Proposition~2.4.1]{Izadi-A4-1995} for \cref{subordinateCase}.
	Using \cref{tripleIntersection},
	\[
		\T_{[p, q]} \cap \T_{[p, r]} \cap \T_{[q, r]} = (\T_{[p, r']} \cap \T_{[p, q']} \cap \T) + [q, r] = (W_p \cup V_{p + q' + r'}) + [q, r] = (W_p + [q, r]) \cup V_{p + q + r},
	\]
giving \cref{shiftedIntersection}.
	The algebraic equivalence $\T \sim \T_{[q, r]}$ on $\prym{2 g}{X}$ restricts to
	\[
		\T \cap \T_{[p, q]} \cap \T_{[p, r]} \sim \T_{[q, r]} \cap \T_{[p, q]} \cap \T_{[p, r]}
	\]
	on $\T_{[p, q]} \cap \T_{[p, r]}$.
	Using \cref{tripleIntersection,shiftedIntersection}, this gives \cref{algequivalent}.
\end{proof}

\begin{lemma}\label{surface isomorphism}
	Suppose $X \tet Y$ are good Prym-curves for $(A, \T)$.
	If $p \in \tX$ corresponds to $P \in \tY^{(4)}$, there is an isomorphism $\psi : \prym{2 g}{X} \ra \prym{2 g}{Y}$ for which $\psi(W_p) = V_P$.
\end{lemma}
\begin{proof}
	Let $L \in \prym{2 g}{Y}$ and $\varphi : \prym{0}{X} \to \prym{0}{Y}$ be as in \cref{tetragonal}.
	Choose theta characteristics $L_X$ and $L_Y$ on $\tX$ and $\tY$ respectively, and define $\psi$ by the formula
	\[
		\psi(M) \eq \varphi(M - L_X) + L_Y.
	\]
	We may choose $L_Y$ so that $\psi$ maps $\Tprym{2 g}(X)$ onto $\Tprym{2 g}(Y)$.
	For $M \in V_P$ and $q \in \tX$ (corresponding to $Q \in \tY^{(4)}$),
	\[
		\varphi(\iota(q) - \iota(p)) + M = \varphi [q, p'] + M = M(Q - P) \in \Tprym{2g}(Y).
	\]
	(note that $M(-P)$ is effective).
	Applying $\psi^{-1}$ shows that
	\[
		\iota(q) - \iota(p) + \psi^{-1}(M) \in \Tprym{2 g}(X)
	\]
	for all $q \in \tX$, which means $\psi^{-1}(M) - \iota(p) \in \tX_\lambda$.
	Therefore $V_P \subseteq \psi(W_p)$.

	As a special subvariety, $V_P$ has class $\frac{1}{3} [\T]^3 \in H^6(A, \bZ)$ \cite[Th\'eor\`eme 1]{Beauville82}.
	On the other hand $V_{p + q + r}$ has class $\frac{2}{3} [\T]^3$ whenever $q, r \in \tX$. Hence, by \cref{manyIntersections}\cref{tripleIntersection}, $\psi_* [W_p] = [V_P]$.
\end{proof}

\subsection{$\lambda$-classes}
Let $X$ be a Prym-curve for $(A, \T)$.
By definition, residuation sends $W_p$ to $W_{p'}$ for all $p \in \tX$.
The algebraic equivalence class of $W_p$ in $\T_{2 g}$ is independent of $p \in \tX$, and therefore fixed by residuation.
Given an isomorphism $\psi : \prym{0}{X} \to A$, there is a unique theta characteristic $L_0$ on $\tX$ such that $\psi(\T_{2 g} - L_0) = \T$.
Let $[\tX_\lambda]$ be the algebraic equivalence class of $\psi(W_p - L_0) \subset \T$; it is independent of $p \in \tX$ and fixed by $-1$.
If $(A, \T)$ is very general, then $\aut(A, \T) = \set{\pm 1}$, which means $[\tX_\lambda]$ does not depend on $\psi$.

\begin{corollary}\label{triple theta sum}
	If $(X, Y, Z)$ is a tetragonal triple of good Prym-curves for a very general Prym variety $(A, \T)$, then $[\tX_\lambda] + [\tY_\lambda] + [\tZ_\lambda] = [\T]^2$.
\end{corollary}
\begin{proof}
	If $p + q + r + s \in \tY \subset \tX^{(4)}$ lifts a reduced divisor of the $g_4^1$ on $X$, and $\Xi \eq \T_{2 g}(X)$, then
	\[
		\Xi \cap \Xi_{[p, q]} \cap \Xi_{[p, r]} = W_p \cup V_{p + q + r + s} \cup V_{p + q + r + s'}
	\]
	by \cref{manyIntersections} parts \cref{tripleIntersection,subordinateCase}.
	Choosing isomorphisms as needed, \cref{surface isomorphism} implies that the classes of (the images of) $W_p$, $V_{p + q + r + s}$ and $V_{p + q + r + s'}$ in $\T$ are $[\tX_\lambda]$, $[\tY_\lambda]$ and $[\tZ_\lambda]$ respectively.
\end{proof}

\section{27 surfaces}\label{secSurfaces}

In this section $g = 5$ and the $\tX_\lambda$ are surfaces.
The fiber of $\prymap$ at a general ppav $(A, \T) \in \cA_5$ consists of $27$ Prym-curves, and the tetragonal correspondence between them is isomorphic to the incidence correspondence for the lines on a smooth cubic surface \cite[4.2]{donagi92}.
This is almost enough to compute the intersection pairing between the $[\tX_\lambda]$.

\subsection{27 lines}\label{27 lines}
The lines in the cubic surface obtained by blowing up $6$ points $p_1, \dots, p_6 \in \bP^2$ in general position (with respect to lines and conics) can be described as follows:
\begin{itemize}
\item
	The exceptional divisor $E_i$ over each point $p_i$.
\item
	The proper transform $F_{i j}$ of the line joining $p_i$ to $p_j$, for $i < j$.
	If $i > j$ we set $F_{i j} \eq F_{j i}$.
\item
	The proper transform $G_j$ of the conic containing $p_i$ for $i \ne j$.
\end{itemize}
Two lines meet \iffw{} they both belong to a triple of the form $(E_i, F_{i j}, G_j)$ with $i \ne j$, or $(F_{i j}, F_{k l}, F_{m n})$ with $\set{i, j, k, l, m, n} = \set{1, \dots, 6}$.
The automorphism group $W(\esix)$ of this configuration acts transitively on lines and sixers (i.e., sextuples of mutually skew lines) \cite[Proposition 9.1.4]{Dolgachev2012}.
Any permutation of the indices acts on the sixer $(E_1, \dots, E_6)$ and on the cubic surface.

\begin{theorem}\label{intersection numbers}
	Let $(A, \T)$ be a very general ppav of dimension $5$.
	Suppose that $[\tX_\lambda]^2 = 16$ whenever $X$ is a Prym-curve for $(A, \T)$.
	If $X$ and $Y$ are non-isomorphic Prym-curves, then
	\[
		[\tX_\lambda] . [\tY_\lambda] = \begin{dcases}
			12 & \text{ if } X \tet Y, \\
			14 & \text{ otherwise.}
		\end{dcases}
	\]
\end{theorem}
\begin{proof}
	First, suppose that $X \tet Y$ and form the tetragonal triple $(X, Y, Z)$.
	Let $\pair{-, -}$ be the pairing on (formal sums of) Prym-curves induced by the intersection form on $\T$, so that
	\[
		\pair{X, X + Y + Z} = [\tX_\lambda] . ([\tX_\lambda] + [\tY_\lambda] + [\tZ_\lambda]) = [\tX_\lambda] . [\T]^2 = \frac{g!}{3} = 40,
	\]
	by \cref{triple theta sum} and the fact that $[\tX_\lambda] = \frac{1}{3} [\T]^3$ in $A$ (see \cref{surface isomorphism}).
	Permuting $X,Y,Z$, we obtain
	\[
		\pair{X, Y + Z} = \pair{Y, Z + X} = \pair{Z, X + Y} = 40 - 16 = 24,
	\]
	so $\pair{X, Y} = \pair{X, Z} = \pair{Y, Z} = 12$.
	
	Now, suppose that $X \not\tet Y$. Choose a Prym-curve $Z_1$ tetragonally related to $X$ and $Y$.
	As above, we form triples $(X, Z_1, X_1)$ and $(Y, Z_1, Y_1)$ so that
	\[
		40 = \pair{X, Y + Z_1 + Y_1} = \pair{X_1, Y + Z_1 + Y_1} = \pair{Y, X + Z_1 + X_1} = \pair{Y_1, X + Z_1 + X_1}
	\]
	and hence
	\begin{equation}\label{tripledouble}
		\pair{X, Y} = 28 - \pair{X_1, Y} = \pair{X_1, Y_1} = 28 - \pair{X, Y_1}.
	\end{equation}
	This does not immediately give $\pair{X, Y} = 14$, but we can find more equations of the form \cref{tripledouble} by varying $X$, $Y$ and $Z_1$.
	For this, it helps to label the Prym-curves as in \cref{27 lines}.
	We have similar identities for all the $4$-tuples of surfaces obtained, as above, from a pair of skew lines in the cubic surface.
	For distinct indices $i,j,k$, we have the following pairs of tetragonally related triples
	\begin{tikzcenter}[node distance = 7ex]
		\node (Fija) {$F_{i j}$};
		\node[above left = of Fija] (Eia) {$E_i$};
		\node[above right = of Fija] (Gia) {$G_i$};
		\node[below left = of Fija] (Gja) {$G_j$};
		\node[below right = of Fija] (Eja) {$E_j$};

		\draw[-]
			(Eia) edge (Fija)
			(Fija) edge (Gja)
			(Gja) edge (Eia)

			(Gia) edge (Fija)
			(Fija) edge (Eja)
			(Eja) edge (Gia);

		\node[xshift = 4cm] (Gkc) {$G_k$};
		\node[above left = of Gkc] (Eic) {$E_i$};
		\node[above right = of Gkc] (Ejc) {$E_j$};
		\node[below left = of Gkc] (Fikc) {$F_{i k}$};
		\node[below right = of Gkc] (Fjkc) {$F_{j k}$};

		\draw[-]
			(Eic) edge (Gkc)
			(Gkc) edge (Fikc)
			(Fikc) edge (Eic)

			(Ejc) edge (Gkc)
			(Gkc) edge (Fjkc)
			(Fjkc) edge (Ejc); 
			
		\node[xshift = 8cm] (Ekb) {$E_k$};
		\node[above left = of Ekb] (Gib) {$G_i$};
		\node[above right = of Ekb] (Gjb) {$G_j$};
		\node[below left = of Ekb] (Fikb) {$F_{ik}$};
		\node[below right = of Ekb] (Fjkb) {$F_{jk}$};

		\draw[-]
			(Gib) edge (Ekb)
			(Ekb) edge (Fikb)
			(Fikb) edge (Gib)

			(Gjb) edge (Ekb)
			(Ekb) edge (Fjkb)
			(Fjkb) edge (Gjb);
	\end{tikzcenter}
	The application of \cref{tripledouble} gives
	\[
		\pair{E_i, G_i} = \pair{G_j, E_j} = 28 - \pair{E_i, E_j} = 28 - \pair{G_i, G_j},
	\]
	\[
		\pair{E_i, E_j} = \pair{F_{ik}, F_{jk}} = 28 - \pair{E_i, F_{jk}} = 28 - \pair{E_j, F_{ik}},
	\]
	\[
		\pair{G_i, G_j} = \pair{F_{ik}, F_{jk}} = 28 - \pair{G_i, F_{jk}} = 28 - \pair{G_j, F_{ik}}.
	\]
It follows from the equations above that all intersection numbers of the form $\pair{E_i, E_j}$, $\pair{G_i, G_j}$ or $\pair{F_{ik}, F_{jk}}$ are equal and all intersection numbers of the form $\pair{E_i, G_i}$, $\pair{E_i, F_{jk}}$ or $\pair{G_i, F_{jk}}$ are equal. Now note that we also have the surfaces obtained from the tetragonally related triple $(F_{12}, F_{34}, F_{56})$. Intersecting the sum of these three surfaces with $E_1$, we obtain:
\[
\pair{E_1, F_{34}} + \pair{E_1, F_{56}} = 28.
\]
Since $\pair{E_1, F_{34}}=\pair{E_1, F_{56}}$, we obtain
\[
\pair{E_1, F_{34}} = 14
\]
and hence all the other intersection numbers above are also equal to $14$.
\end{proof}

\section{27 curves}\label{secprymcurves}

For a general ppav $B$ of dimension $g = 4$, the fiber of $\prymap$ over $B$ is two-dimensional, a double cover of the Fano surface of lines on a particular cubic threefold \cite[5.1]{donagi92}, \cite[6.27]{Izadi-A4-1995} .
The covering involution is given by $X \mapsto X_\lambda$, where $\tX_\lambda \to X_\lambda$ is the quotient by residuation \cite[3.11, 3.13]{Izadi-A4-1995}.
Each point $\alpha \in B$ determines a hyperplane section of the cubic threefold, and a choice of Prym-curve $X$ lying over each of its lines \cite[3.16, 4.8, 4.9, 5.7, 5.10]{Izadi-A4-1995}.
The curves $\tX_\lambda$ admit embeddings in $\Xi \cap \Xi_\alpha$, where $\Xi$ is a theta divisor for $B$.
In this section we prove an analog of \cref{intersection numbers} for these embeddings.

For $0 \ne \alpha \in B$, one has $W_p \subset \T \cap \T_\alpha$ \iffw{} $\alpha = [p, q]$ for some $q \in \tX$ \cite[3.16]{Izadi-A4-1995}.

\begin{definition}
An $\alpha$-curve is a Prym-curve $X$ for $B$ such that $W_p \subset \T \cap \T_\alpha$ for some $p \in \tX$, where $\T \eq \Tprym{8}(X)$ is the canonical translate of $\Xi$ in $\prym{8}{X}$.
\end{definition}

\begin{lemma}\label{betterAlgEquivalence}
	If $X$ is an $\alpha$-curve such that $X$ and $X_\lambda$ are good, then the translates of $\tX_\lambda$ in $\Xi \cap \Xi_\alpha$ are algebraically equivalent.
\end{lemma}
\begin{proof}
	Pick $q, r \in \tX$ such that $\alpha = [q', r]$.
	The two embeddings of $\tX_\lambda$ in $\T \cap \T_\alpha$ are $W_{q'}$ and $W_r$.
	Given $p \in \tX \minus \set{q, q', r, r'}$, it follows that the embeddings of $\tX_\lambda$ in $(\T \cap \T_\alpha) + [p, q] = \T_{[p, q]} \cap \T_{[p, r]}$ are $W_{q'} + [p, q] = W_p$ and $W_r + [p, q] = W_p + [q, r]$, which are algebraically equivalent by \cref{manyIntersections}\cref{algequivalent}.
	Now translate back to $\Xi \cap \Xi_\alpha$.
\end{proof}

\subsection{Curve classes}
As a consequence there is a well-defined class $[\tX_\lambda] \in H^2(\Xi \cap \Xi_\alpha, \bZ)$.
It is invariant under the involution $\beta \mapsto \alpha - \beta$ on $\Xi \cap \Xi_\alpha$.
Since $\gamma_X \eq [\tX_\lambda] - \frac{1}{3} [\Xi]\in H^2(\Xi \cap \Xi_\alpha, \bQ)$ pushes forward to $0$ in $B$, it belongs to $\bW_\bQ^+$, where $\bW \subset H^2(\Xi \cap \Xi_\alpha, \bZ)$ is the primal cohomology for $\Xi \cap \Xi_\alpha \into B$.
Kr\"amer showed that $\bW^+ \iso \esix(-2)$ \cite[5.1]{Kramer2015}, and that $\gamma_X$ is a norm-minimizing element of the dual lattice \cite[Lemma~7.2]{Kramer2015}.
We refine his calculation using the argument of \cref{intersection numbers}, after establishing the following Lemma.

\begin{lemma}\label{modelCurves}
	Suppose $B$ is very general and $\alpha \in B$ is general.
	\begin{anumerate}
	\item\label{27 embeddings}
		If $X$ is an $\alpha$-curve, then $X$ and $X_\lambda$ are good.
	\item\label{triple breakdown}
		For $\alpha$-curves $X \tet Y$, the other curve in the tetragonal triple $(X, Y, Z)$ is also an $\alpha$-curve, and $[\tX_\lambda] + [\tY_\lambda] + [\tZ_\lambda] = [\Xi]$ in $\Xi \cap \Xi_\alpha$.
	\end{anumerate}
\end{lemma}
\begin{proof}
	Since $B$ is general, the fiber of $\prymap$ over $B$ contains a dense open subset of good Prym-curves \cite[3.3, 3.10]{Izadi-A4-1995}.
	The complement $Z$ is therefore (at most) one-dimensional.
	Recall that $X$ is an $\alpha$-curve if and only if $\alpha$ belongs to the surface $\Sigma(X) \eq \set{[p, q] \mid p, q \in \tX}$ \cite[3.16]{Izadi-A4-1995}.
	For dimension reasons, we may assume that $\alpha \notin \cup_{X \in Z \cup \lambda(Z)} \Sigma(X)$, which gives \cref{27 embeddings}.

	For any $\alpha \in B$ and Prym-curves $X \tet Y$ such that $\alpha \in \Sigma(X) \cap \Sigma(Y)$, call $\alpha$ \emph{good} for $X \tet Y$ if there is a reduced divisor $p + q + r + s \in \tY$ such that $\alpha = [p, q]$.
	According to \cite[5.9]{Izadi-A4-1995}, $\Sigma(X) \cap \Sigma(Y)$ has (possibly impure) dimension one, and a general point on it is good for $X \tet Y$.
	Hence the set of elements of $\Sigma (X)$ which are not good for $X \tet Y$ for some $Y$ is of dimension $\leq 1$. The union of these sets for all Prym-curves $X$ is at most a threefold in $B$. So a general $\alpha \in B$ is good for all pairs $X \tet Y$ such that $\alpha \in \Sigma(X) \cap \Sigma(Y)$.

	Given $\alpha$-curves $X \tet Y$, we may pick $p + q + r + s \in \tY$ as above, and form the tetragonal triple $(X, Y, Z)$.
	The argument of \cref{triple theta sum} completes the proof of \cref{triple breakdown}.
\end{proof}

\begin{proposition}\label{intersection numbers redux}
	Suppose $\alpha \in B$ is general.
	If $X$ and $Y$ are $\alpha$-curves, then
	\[
		[\tX_\lambda] . [\tY_\lambda] = \begin{dcases}
			0 & \text{ if } X = Y \\
			4 & \text{ if } X \tet Y, \\
			2 & \text{ otherwise.}
		\end{dcases}
	\]
\end{proposition}
\begin{proof}
	Since $\alpha$ is general, $\Xi \cap \Xi_\alpha$ is smooth, so the adjunction formula shows that $[\tX_\lambda] . [\tX_\lambda] = 0$.
	
	For the other cases, suppose that $B$ is very general.
	The tetragonal correspondence between $\alpha$-curves is isomorphic to the incidence correspondence between the $27$ lines on a smooth cubic surface \cite[5.9, 5.12, 6.8]{Izadi-A4-1995}.
	We may therefore argue as in \cref{intersection numbers} using \cref{modelCurves}.

	Since intersection numbers are constant in smooth families, the result holds even when $(B, \Xi)$ has extra automorphisms (but remains sufficiently general).
\end{proof}

\begin{corollary}\label{Wgenerators}
	If $\alpha \in B$ is general, then the invariant primal cohomology $\bW^+ \subset H^2(\Xi \cap \Xi_\alpha, \bZ)$ is freely generated by the classes $\delta_i \eq [\tX_{i \lambda}] - [\tY_\lambda]$, for any collection of $\alpha$-curves $X_1, \dots, X_6, Y$ such that the $X_i$ are mutually tetragonally unrelated and $Y$ is related to exactly two of the $X_i$.
\end{corollary}
\begin{proof}
	If $X$ and $Y$ are $\alpha$-curves corresponding to lines $E$ and $F$ on a cubic surface, then \cref{intersection numbers redux} says that $[\tX_\lambda] . [\tY_\lambda] = 2 (E . F + 1)$.
	Thus, the matrix with entries $\delta_i . \delta_j$ is a Gram matrix for $\esix(-2)$.
	In particular, its determinant is the discriminant of $\bW^+$, namely $192$. 
	The map $\bZ^6 \to \bW^+$ defined by the $\delta_i$ has to be invertible for this to hold.
\end{proof}

\section{The one-parameter family of abelian fivefolds}\label{secfamily}

We summarize the construction of the family of theta divisors from \cite{IzadiWang2018}.
The goal is to produce a rank one degeneration with smooth total space.

\subsection{Enriques and K3 surfaces}
Let $R$ be a very general Enriques surface and $\rho : \tR \to R$ the K3 \'etale double cover corresponding to the canonical class (which is $2$-torsion) $K_R \in \Pic(R)$.
Let $H$ be a very ample line bundle on $R$ with $H^2 = 10$.
A general element in the linear system $\abs{H} \iso \bP^5$ is a smooth curve of genus $6;$ such curves are parameterized by the Zariski open subset $\abs{H} \minus D$, where $D$ is the dual variety of the embedding of $R$ in $\dual{\abs{H}}$.
For each curve $C \in \abs{H} \minus D$, we obtain a nontrivial \'etale double cover $\tC \eq \rho^{-1}(C) \to C$.
Associating to such a cover its Prym variety $\prym{0}{C}$ defines a morphism from $\abs{H} \minus D$ to $\cA_5$:
\begin{tikzcenter}[node distance = 10ex and 4em]
	\node (H) {$\abs{H} \minus D$};
	\node[below right = of H] (A5) {$\cA_5$\nowidth{.}};
	\node[above right = of A5] (R6) {$\cR_6$};

	\draw[-To]
		(H) edge (A5)
		(H) edge node {$\mu_H$} (R6)
		(R6) edge node {$\prymap$} (A5);
\end{tikzcenter}
The linear systems $\abs{H}$ form a projective bundle $\cH$ over the moduli space of Enriques surfaces, and the maps $\mu_H$ are restrictions of a rational map $\mu : \cH \to \cR_6$.
Mori and Mukai \cite{MoriMukai83} showed that $\mu$ is dominant.

\subsection{A family of curves}
Suppose $T \iso \bP^1 \subset \abs{H}$ is a Lefschetz pencil, i.e., it is transverse to the dual variety $D$.
Then the singular curves of the pencil consist of finitely many irreducible nodal curves.
Denote by $\cC \eq \Bl_{10} R$ (resp. $\tcC \eq \Bl_{20} \tR$) the blow-up of $R$ (resp.\ $\tR$) along the base locus of $T$ (resp.\ $\rho^* T$).
We obtain a family of irreducible \'etale double covers parameterized by $T$:
\begin{tikzcenter}[node distance = 10ex and 4em]
	\node (tC) {$\tcC$};
	\node[below right = of tC] (T) {$T$\nowidth{.}};
	\node[above right = of T] (C) {$\cC$};

	\draw[-To]
		(tC) edge (T)
		(tC) edge node {$\pi$} (C)
		(C) edge (T);
\end{tikzcenter}

\subsection{Singular fibers}\label{normalized fibers}
There are $42$ points $t_1, \dots, t_{42} \in T$ where the fiber $C_i \eq \cC_{t_i}$ of the family $\cC \to T$ is singular.
The Prym varieties of the double covers in this family are well-defined as semi-abelian varieties.
They are abelian at the smooth fibers and have a rank $1$ toric part at the singular fibers.
This family of Prym varieties and their theta divisors can be compactified to families
\begin{tikzcenter}[node distance = 10ex and 4em]
	\node (Th) {$\T$};
	\node[below right = of Th] (T) {$T$};
	\node[above right = of T] (A) {$A$};

	\draw[-To]
		(Th) edge (T)
		(A) edge (T);

	\draw[In-To]
		(Th) edge (A);
\end{tikzcenter}
with smooth total spaces.
The fibers $\T_i \subset A_i$ over $t_i$ have the following descriptions.

Denote by $(B_i, \Xi_i)$ the Prym variety of $\tX_i \to X_i$, where $X_i$ (resp.\ $\tX_i$) is the normalization of $C_i$ (resp.\ $\tC_i$).
It is a ppav of dimension $4$.
The semi-abelian Prym variety of $\tC_i \to C_i$ is an extension of $B_i$ by $\bC^\times$.
Such an extension is determined by a point $\beta \in B_i$, well-defined up to $\pm 1$.
The normalization $P_i$ of $A_i$ is the $\bP^1$-bundle over $B_i$ associated to this semi-abelian variety, and $A_i$ is obtained from $P_i$ by identifying the $\infty$-section $B_\infty \iso B_i$ with the $0$-section $B_0 \iso B_i$ after translation by $\beta$.
The normalization of $\T_i$ is the blow-up $\tB_i \to B_i$ along $\Xi_i \cap \Xi_{i \beta}$.
To recover $\T_i$ from $\tB_i$, one glues the proper transforms of $\Xi_i$ and $\Xi_{i \beta}$ after translation by $\beta$.

\subsection{Correcting the degree}
In order to construct surfaces in a fiber $\T_t$, it it more natural to work in degree $10$, i.e., we should have $A_t = \prym{10}{\cC_t}$ and $\T_t = \T(\cC_t)$.
The construction of these families is similar to the degree zero case \cite[{\S}3.5]{ArbarelloFerrettiSacca}.
To check that their total spaces remain smooth, one can use \'etale-local (or local analytic) sections of $\T \to T$ to show that the families in degree $0$ and $10$ are \'etale-locally (or locally analytically) isomorphic.
Alternatively, if one is willing to throw away the smooth fibers $A_t$ for which $\T_t$ is singular (which is harmless for our purposes), it is easy to prove directly (as in \cite[2.5]{IzadiWang2018}) that $\T$ is nonsingular.

\section{Families of surfaces}\label{the family of surfaces}

We construct families of special surfaces in the fibers of $\T \to T$.

\subsection{Nets of degree $6$}
Suppose $C \in \ocM_6$ is irreducible and has at most one node.
Given a $g^2_6$ on $C$, i.e., a net of degree $6$, and an \'etale double cover $\tC \to C$, one can define special surfaces $S \subset \tC^{(6)}$ as in the introduction.
The relative version of this does not work over $T$ because the fibers of $\cC \to T$ have no canonical choice of $g_6^2$.
We will fix this by passing to a new base $U$; the map $U \to T$ should be unramified in order to preserve the smoothness of $\T$.
There is a non-empty Zariski open subset $\cU^2_6 \subset \ocM_6$ parameterizing curves $C$ with $\card{G_6^2(C)} = 5$.
Let $\Delta_0 \subset \ocM_6$ be the boundary component whose generic points parameterize irreducible curves.

\begin{lemma}\label{five g26}
The intersection $\cU^2_6 \cap \Delta_0$ is not empty.
\end{lemma}
\begin{proof}
	Suppose $C \in \Delta_0$ is general and let $X \tou\nu C$ be its normalization.
	Let $p, q \in X$ be the points over the node $r \in C$.
	Recall that $\Pic^6(C) \tou{\nu^*} \Pic^6(X)$ is a $\bC^\times$-torsor with fiber
	\[
		\bP(\res{L}{p} \oplus \res{L}{q}) \minus \set{\res{L}{p} \oplus 0, 0 \oplus \res{L}{q}}
	\]
	over $L \in \Pic^6(X)$ \cite[Corollary~12.4]{OdaSeshadri1979}.
	Given $L' \in \Pic^6(C)$ with $\nu^* L' \iso L$, the corresponding line in $\res{L}{p} \oplus \res{L}{q}$ is the kernel of the subtraction map $\res{L}{p} \oplus \res{L}{q} \to \res{L'}{r}$.
	This map is the evaluation at the node of the surjection in the following exact sequence, obtained by tensoring $\cO_C \to \nu_* \cO_X$ with $L'$:
	\begin{equation}\label{twisted nodal curve sequence}
		0 \to L' \to \nu_* L \to \res{L'}{r} \to 0.
	\end{equation}

	If $h^0(L') = 3$ then \cref{twisted nodal curve sequence} forces $h^0(L) = 3$ by Clifford's theorem and the genericity of $X$.
	Thus we may identify the maps $H^0(X, L) \to \res{L}{w}$ and $H^0(C, L') \to \res{L'}{\nu(w)}$ for each $w \in X$.
	In particular
	\[
		H^0(X, L(-p - q)) = H^0(X, L(-p)) = H^0(X, L(-q))
	\]
	is two-dimensional (as $X$ is neither hyperelliptic nor trigonal, it has no $g_5^2$).

	Conversely, if $h^0(L(-p - q)) = h^0(L) - 1 = 2$, then exactly one $L'$ lying over $L$ satisfies $h^0(L') = 3$.
	Indeed, if $h^0(L') = h^0(L)$ then the surjection in \cref{twisted nodal curve sequence} vanishes on global sections, i.e., the composition
	\[
		H^0(X, L) \to \res{L}{p} \oplus \res{L}{q} \to \res{L'}{r}
	\]
	is zero.
	Since $h^0(L(-p - q)) = 2$, the above sequence must be exact, so the line corresponding to $L'$ is unique (and it is neither of the summands, because $X$ has no $g_5^2$).

	It remains to show that exactly five $L \in W_6^2(X)$ satisfy $h^0(L(-p - q)) = 2$.
	On a general smooth curve $X$ of genus five, every $g_6^2$ has the form $\abs{\omega_X(-D)}$ for a unique $D \in X^{(2)}$.
	The pair $(X, \omega_X(-p - q))$ is general, so the image of $X$ in $\dual{\abs{\omega_X(-p - q)}}$ is a nodal sextic \cite[XXI, {\S}10]{ACG}.
	By the genus formula it has five nodes, so there are five $D \in X^{(2)}$ such that $h^0(\omega_X(-D - p - q)) = 2$, giving five choices for $L \eq \omega_X(-D)$.
\end{proof}

\subsection{First base change}
If $T \subset \abs{H}$ is sufficiently general, then \cref{five g26} implies that there is a Zariski open subset $T^\circ \subset T$ such that $\cC_t \in \cU_6^2$ for all $t \in T^\circ$ and $C \eq \cC_0$ is singular for a unique point $0 \in T^\circ$.
Let $X \tou{\nu} C$ be the normalization.
By an argument analogous to \cite[Lemma 2.3]{IzadiWang2018}, we may assume that the five pairs $(X, \nu^* g_6^2)$ are also general.
Let $U \eq \cG_6^2(\cC_{T^\circ})$ parameterize nets of degree $6$ on the fibers of $\cC_{T^\circ} \eq \cC \times_T T^\circ \to T^\circ$.
It is \'etale of degree $5$ over $T^\circ$.

\subsection{Surfaces of divisors}
Each fiber of the family $\cC_U \to U$ has a canonical $g_6^2$.
These linear systems form a $\bP^2$-bundle $\cP \to U$ embedded in the relative symmetric
power $\cC_U^{(6)}$.  Put $\cS \eq \cP \times_{\cC_U^{(6)}} \tcC_U^{(6)}$.
%
%

\subsection{Second base change}
The map $\cS \to U$ naturally factors as $\cS \to \tU \to U$, where $\tU$ is the double cover of $U$ parameterizing the connected components of the fibers of $\cS \to U$.
From now on we think of $\cS$ as a family over $\tU$; for instance $\cS_u \eq \cS \times_\tU \set{u}$ is connected for $u \in \tU$.
The following fact ensures that $\tU \to U$ is unramified, so that $\T_\tU$ is nonsingular:

\begin{lemma}\label{lemadmissible}
	For every $g_6^2$ on $C$, the corresponding fiber of $\cS$ in $\tC^{(6)}$ has two connected components.
\end{lemma}
\begin{proof}
	One checks (e.g., using \cref{twisted nodal curve sequence}) that the surface in $\tC^{(6)}$ is the image of the surface in $\tX^{(6)}$ determined by $\nu^* g_6^2$, which has two connected components.
	These components are smooth, and their images in $\tC^{(6)}$ are disjoint, provided that the image of $C \to \dual{g_6^2}$ is admissible in the sense of Welters \cite[(8.14), (9.2), (9.6)]{Welters81}. This means that the image of $C$ is nodal and that for any line $l$ in $\dual{g^2_6}$, the divisor cut on $l$ by (the image of) $C$ is either reduced or contains exactly one divisor of the form $2P, 3P$, or $2P+2Q$.
	It is easy to see that a general plane sextic of geometric genus five is admissible \cite[1(c)]{DiazHarris1988}.
	Since the five pairs $(X, \nu^* g_6^2)$ are general, the result follows.
\end{proof}

\begin{lemma}\label{irreducible surfaces}
	Given $u \in \tU$ lying over $0 \in T$, there is a connected Zariski open neighborhood $\tU^\circ \subseteq \tU$ of $u$ such that $\cS_v$ is irreducible for all $v \in \tU^\circ$, and smooth unless $v$ lies over $0$.
	Consequently $\cS_{\tU^\circ}$ is integral.
\end{lemma}
\begin{proof}
	By \cite[(8.13)]{Welters81}, if the image of $X \to \dual{g_6^2}$ is admissible (see the proof of \cref{lemadmissible}), then the special surfaces obtained from $g^2_6$ in $\tX^{(6)}$ are smooth. As in the proof of \cref{lemadmissible}, general nodal plane sextics are admissible.
	The same argument applies to a general curve of genus $6$ together with a $g_6^2$, so $\cS_v$ is smooth for most $v \in \tU$ provided that $T$ and $R$ are sufficiently general. We let $\tU^\circ$ be the connected component of $\tU$ containing $u$ minus the points $v$ where $\cS_v$ is singular.

	Since $\cS \to \cP$ is finite, flat and generically \'etale, $\cS$ satisfies Serre's
	conditions $R_0$ and $S_1$ (in fact it is Cohen-Macaulay), so $\cS$ is reduced.
	The previous paragraph then implies that $\cS_{\tU^\circ}$ is integral.
\end{proof}

\subsection{Third base change}
In order to embed the surfaces we constructed into the theta divisors of our family, we need to choose lifts of divisors of the residual $g_4^1$ that we will then add to lifts of the divisors of the $g^2_6$ to obtain lifts of canonical divisors. For this we need to introduce a third base change, defined as follows. Taking the $g_4^1$ residual to each $g_6^2$ determines, in a completely analogous way, a family
$\cD \to \tU$ of 1-dimensional special subvarieties in $\tcC_\tU^{(4)}$. 
The morphism $C \to \dual{g_4^1}$ determined by the image $g^2_6 = K_C -g_4^1 \in U$ of $u \in \tU$ is generically unramified, so $\cD_u \to g_4^1$ is also generically unramified, which means $\cD_u$ is generically smooth.
Given a general divisor $D \in \cD_u$, there is a connected curve $\tT \subset \cD_{\tU^\circ}$ containing $D$ such that $\tT \to \tU^\circ$ is \'etale.
Such a curve can be obtained, for instance, by choosing a hyperplane section of $\cD$ (in some projective embedding) which meets the smooth locus of $\cD_u$ transversely, then removing the ramification locus from a component which contains $D$.

\subsection{The family of surfaces}\label{par-cV}
The natural embedding of $\cS_\tT$ in $\tcC_\tT^{(10)}$ induces a rational map $\cS_\tT \to \T_\tT$, defined on the open subset $\cS^\circ \subset \cS_\tT$ of divisors avoiding the nodes.
Note that $\cS_E^\circ \ne \emptyset$ for all $E \in \tT$, and $\cS_E^\circ = \cS_E$ unless $E$ lies over $0 \in T$.
There is a dense open subset $\tT^\circ \subset \tT$ such that the image of the birational morphism $\cS_t \to \T_t$ is smooth for all $t \in \tT^\circ$.
Since $\cS \to U$ is flat, the image of $\cS_{\tT^\circ} \to \T_{\tT^\circ}$ is flat over $\tT^\circ$.
The (scheme-theoretic) closure $\cV$ of the image of $\cS^\circ$ in $\T_\tT$ agrees with that of $\cS_{\tT^\circ}$, because $\cS^\circ$ is integral.
In particular $\cV \to \tT$ is flat \cite[III, 9.8]{Hartshorne77}.

\section{A nodal fiber}\label{nodal fiber}

In this section we determine the fiber $V \subset \T_0$ of $\cV$ (see \cref{par-cV}) over $D \in \tT$.

\subsection{Notation}\label{DEF def}
Let $P \to A_0$ and $\tB \to \T_0$ be the normalizations.
Recall from \cref{normalized fibers} that $P \to B$ is a $\bP^1$-bundle with distinguished sections $B_0$ and $B_\infty$, and $\tB \to B$ is the blowup along $\Xi \cap \Xi_\beta$ for the extension datum $\beta \in B$.
If $p, q \in \tX$ lie over one of the nodes of $\tC$, we may assume that $\beta = [p, q']$ (the other choice $[p', q]$ corresponds to relabelling $B_0$ and $B_\infty$) \cite[Proposition~1.5]{IzadiWang2018}.

There is a unique divisor in the pencil $\abs{\pi_* D}$ passing through the node of $C$;
the corresponding divisor on the normalization $X$ can be written as $\pi_* E$ where $E = p + q' + u + v$ for some $u, v \in \tX$.
Since $D$ is general it can be identified with a divisor on $\tX$.
Replacing $v$ by $v'$ if necessary, we may assume that $D$ and $E$ belong to the same curve $\tY
\subset \tX^{(4)}$ (among the two curves tetragonally related to $\tX$ via $\abs{\pi_* D}$).
The divisor $F \eq p + q' + u' + v'$ also belongs to $\tY$. Let $d, e, f \in \tY$ correspond to $D, E, F \in \tX^{(4)}$. If $g_Y$ denotes the $g^1_4$ on $Y$, then $h^0 (g_Y - \pi_* (e+f)) >0$.

\begin{lemma}
	The special subvariety $V_{d + e + f}$ is smooth.
\end{lemma}
\begin{proof}
	Set $g_5^1 \eq \abs{\omega_Y(-\pi_*(d + e + f))}$, and suppose for a moment that $Y \to \dual{g_5^1}$ is a well-defined morphism with only simple ramification.
	By Welters' criterion \cite[(8.13)]{Welters81}, the associated special subvarieties $S_1, S_2 \subset \tY^{(5)}$ are smooth.
	We may assume that $V_{d + e + f}$ is the image of $S_1$.
	If some pencil in $\tY^{(5)}$ meets $S_1$, its image in $Y^{(5)}$ must be $g_5^1$, so the pencil must be all of $S_1$.
	This is not possible: if a divisor $p_1 + \ldots + p_5$ belongs to $S_1$, so does $p_1' + p_2' + \ldots + p_5$, hence $p_1 + p_2$ is linearly equivalent to $p_1' + p_2'$ and $\tY$ is hyperelliptic, which contradicts our genericity assumptions.
	Therefore $S_1$ maps isomorphically onto $V_{d + e + f}$.

	It remains to show that $Y \to \dual{g_5^1}$ has simple ramification.
	For this, we just need the pair $(Y, g_5^1)$, or equivalently $(Y, \pi_*(d + e + f))$, to be sufficiently general \cite[XXI, (11.9)]{ACG}.

The data $(X, \pi_* (p+q))$ is in finite correspondence with the data $(Y, \pi_* (e+f))$. So, for a general choice of $(X, \pi_* (p+q))$, $(Y, \pi_* (e+f))$ will also be general. We can then choose $d$ to be a general point on $\tY$.
\end{proof}

\subsection{The central fiber}
We will show that $V$ is birational to
\[
	W \eq \Xi_{[d, e]} \cap \Xi_{[d, f]} = (\Xi \cap \Xi_{[u, v]}) + [d, f] \subset B
\]
(see \cref{tetragonal}). Since $X$ is general, we may assume that $[u, v], [p, q'] \in B$ are too \cite[4.6]{Izadi-A4-1995}.
This implies that $W$ and $\Xi \cap \Xi_\beta$ are smooth by a Bertini-style argument.

\begin{lemma}\label{blowupintersections}
	The projection $\tB \to B$ induces isomorphisms
	\begin{align*}
		\tW &\to W, \\
		\Delta \cap \tW &\to V_{d + e + f}, \\
		\tXi \cap \tW &\to W_d, \\
		\tXi_\beta \cap \tW &\to W_d + \beta,
	\end{align*}
	where $\Delta \subset \tB$ is the exceptional divisor and $\tZ \subset \tB$ denotes the proper transform whenever $Z \subset B$.
\end{lemma}
\begin{proof}
	Note that $W_d$ and $W_d + \beta$ are the embeddings of $\tY_\lambda$ in $W$, as $\beta = [e, f]$ (see \cref{tetragonal}). By \cref{intersection numbers redux}, $W_d \cap (W_d + \beta) = \emptyset$. It follows by \cref{manyIntersections}\cref{tripleIntersection,shiftedIntersection} that
	\[
		\Xi \cap \Xi_\beta \cap W = \Xi \cap \Xi_{[e, f]} \cap \Xi_{[d, e]} \cap \Xi_{[d, f]} = \left(  V_{d + e + f} \cup W_d \right) \cap \left(  V_{d + e + f} \cup W_d +\beta \right) = V_{d + e + f}.
	\]
Next note that
\[
W = \left( W \setminus W_d \right) \cup \left( W \setminus W_d + \beta \right)
\]
and
\[
V_{d + e + f} = \left( V_{d + e + f} \setminus W_d \right) \cup \left( V_{d + e + f} \setminus W_d + \beta \right).
\]
Now $V_{d + e + f} \setminus W_d$ is a Cartier divisor in $W \setminus W_d$, and $V_{d + e + f} \setminus W_d+\beta$ is a Cartier divisor in $W \setminus W_d +\beta$.
	This implies that $\Xi \cap \Xi_\beta \cap W= V_{d + e + f}$ is a Cartier divisor in $W$, giving the first two isomorphisms.

	Moreover $(\tXi \cap \tW) \minus \Delta \to (\Xi \cap W) \minus (\Xi \cap \Xi_\beta) = W_d \minus V_{d + e + f}$ is an isomorphism.
	Since $V_{d + e + f}$ is smooth, $\Xi$ intersects $W$ transversely along $V_{d + e + f} \minus W_d$, which means $\tXi \cap \tW = \tW_d \stackrel{\cong}{\ra} W_d$ as sets and also generically as schemes.
	Next, the isomorphism holds scheme-theoretically everywhere because $\tXi \cap \tW$ is a generically smooth Cartier divisor in the smooth variety $\tW$, so it is reduced.
	The fourth isomorphism is similar.  
\end{proof}

\begin{lemma}\label{imSo}
	$V$ contains the image $\oW \subset \T_0$ of $\tW \subset \tB$.
\end{lemma}
\begin{proof}
	First note that $W = (\Xi \cap \Xi_{[u, v]}) + [d, f] = V_{u + v} + \cO_\tX(D - F')$ by \cref{manyIntersections}\cref{doubleIntersection} and \cref{tetragonal}.
	Thus, a general point $\oL \in \oW$ corresponds to exactly one $L \in W \minus V_{d + e + f}$, which can be represented by
	\[
		G + u + v + D - F' = G + D - p' - q
	\]
	for some divisor $G \in \tX^{(6)}$ supported away from the nodes with $\pi_* G \in \nu^* g_6^2$.
	By abuse of notation we can think of $G$ as a divisor on $\tC$, in which case $\pi_* G \in g_6^2$ and $G + D \in \cS^\circ$.

	Let $\T_0^\circ \iso \tB \minus (\tXi \cup \tXi_\beta)$ be the smooth locus of $\T_0$.
	The map $\T_0^\circ \into \tB \to B$ is induced by $\nu^* : \Pic^0(\tC) \to \Pic^0(\tX)$.
	By choosing appropriate theta characteristics on $\tC$ and $\tX$, it can be taken to send $\cO_\tC(G + D) \in V$ to $L \in W$ (we could have chosen $G + D - p - q'$ instead; however $G + D - p - q$ and $G + D - p' - q'$ have the wrong parity).
	It follows that $\cO_\tC(G + D) = \oL$.
	This shows that $V$ contains an open subset of $\oW$, and hence all of $\oW$.
\end{proof}

\begin{lemma}
	The Hilbert polynomial of $V$ is $\chi(\cO_V(n \T_0)) = 20 n^2 - 40 n + 22$.
\end{lemma}
\begin{proof}
	Since $\cV \to \tT$ is flat, it suffices to compute the Hilbert polynomial of $\cV_t$ for any $t \in \tT^\circ$.
	For simplicity set $Z \eq \cC_t$ and let $G \in \tZ^{(4)}$ correspond to $t,$ so that $\cV_t = V_G$.
	The special surface $S \subset \tZ^{(6)}$ associated to $g_4^1 \eq \abs{\pi_* G}$ has two components; let $S_1$ be the one whose image is $V_G$.
	Since $S_1 \to V_G$ is a birational morphism between smooth varieties, their Hilbert polynomials are the same.
	We will show that $\chi(\cO_S(n \T_t)) = 40 n^2 - 80 n + 44$, which gives the result (after some algebra) because the calculation also works on each of the two curves tetragonally related to $Z$ via $g_4^1$.
	Since $\chi(\cO_S(2 n \T_t)) = \chi(\cO_S(n \Tt))$ for all $n \in \bN$, where $\Tt$ is a theta divisor for $\Pic^6(\tZ)$, it suffices to show that $\chi(\cO_S(n \Tt)) = 160 n^2 - 160 n + 44$.

	Let $\iota : g_6^2 \into Z^{(6)}$ be the embedding of the $g_6^2$ residual to $g_4^1$, with $\rho : S \to g_6^2$ and $\tiota : S \into \tZ^{(6)}$ the associated projections.
	Since $\pi^{(6)} : \tZ^{(6)} \to Z^{(6)}$ is an affine morphism, $\pi_*^{(6)}$ commutes with arbitrary base change.
	In particular $\iota^* \pi_*^{(6)} H = \rho_* \tiota^* H$ for $H \eq \cO_{\tZ^{(6)}}(n \Tt)$.
	Since $R^i \rho_* = 0$ for $i > 0$
	\begin{equation}\label{hilbert polynomial integral}
		\chi(\cO_S(n \Tt)) = \chi(\iota^* \pi_*^{(6)} H) = \int_{\abs{L}} \ch(\iota^* \pi_*^{(6)} H) \todd(\abs{L}) = \int_{\abs{L}} \iota^* \ch(\pi_*^{(6)} H) \todd(\abs{L}).
	\end{equation}
	By Grothendieck-Riemann-Roch
	\begin{equation}\label{grr on symmetric powers}
		\ch(\pi_*^{(6)} H) = \pi_*^{(6)}(\ch(H) \todd(\tZ^{(6)})) \cdot \todd(Z^{(6)})^{-1}.
	\end{equation}
	The Chern classes of symmetric products are well-known \cite[VII, (5.4)]{ACGH}.
	In particular
	\[
		c(\tZ^{(6)}) = 1 - 4\teta - \ttheta + 10\teta^{2} + 5\teta \ttheta + \frac{1}{2}\ttheta^{2} + \cdots,
	\]
	(see \cref{symmetric power etale pushforward} for notation) and hence
	\[
		\todd(\tZ^{(6)}) = 1 - 2\teta - \frac{1}{2}\ttheta + \frac{13}{6}\teta^{2} + \frac{13}{12}\teta \ttheta + \frac{1}{8}\ttheta^{2} + \cdots.
	\]
	Since
	\[
		\ch(H) = 1 + n\ttheta + \frac{1}{2} n^{2}\ttheta^{2} + \cdots,
	\]
	it follows that
	\[
		\ch(H) \todd(\tZ^{(6)}) = 1 - 2\teta + \left(n - \frac{1}{2}\right)\ttheta + \frac{13}{6}\teta^{2} + \left(-2 n + \frac{13}{12}\right)\teta \ttheta + \left(\frac{1}{2} n^{2} - \frac{1}{2} n + \frac{1}{8}\right)\ttheta^{2} + \cdots,
	\]
	so by \cref{symmetric power etale pushforward}
	\begin{align*}
		\pi_*^{(6)}(\ch(H) \todd(\tZ^{(6)}))
		&= 64 + \left(160 n - 144\right)\eta + \left(32 n - 16\right)\theta + \left(160 n^{2} - 320 n + \frac{484}{3}\right)\eta^{2}\\
		&+\left(80 n^{2} - 112 n + \frac{112}{3}\right)\eta \theta + \left(8 n^{2} - 8 n + 2\right)\theta^{2} + \cdots.
	\end{align*}
	Again a general formula gives
	\[
		c(Z^{(6)}) = 1 + \eta - \theta + \frac{1}{2}\theta^{2} + \cdots,
	\]
	and hence
	\[
		\todd(Z^{(6)})^{-1} = 1 - \frac{1}{2}\eta + \frac{1}{2}\theta + \frac{1}{6}\eta^{2} - \frac{1}{3}\eta \theta + \frac{1}{8}\theta^{2} + \cdots.
	\]
	Using \cref{grr on symmetric powers}
	\begin{align*}
		\ch(\pi_*^{(6)} H)
		&= 64 + \left(160 n - 176\right)\eta + \left(32 n + 16\right)\theta + \left(160 n^{2} - 400 n + 244\right)\eta^{2} \\
		&+\left(80 n^{2} - 48 n - 48\right)\eta \theta + \left(8 n^{2} + 8 n + 2\right)\theta^{2} + \cdots.
	\end{align*}
	The class of a linear system in $Z^{(d)}$ can be computed using a special case of the secant plane formula \cite[VIII, (3.2)]{ACGH}.
	After simplifying (using \cite[(6.3)]{Macdonald62}), the class of $\abs{L}$ in $Z^{(6)}$ is
	\[
		10\eta^{4} - 4\eta^{3} \theta + \frac{1}{2}\eta^{2} \theta^{2}.
	\]
	It follows that the degree of $\iota^* \ch(\pi_*^{(6)} H)$ is $160 n^{2} - 400 n + 244$.
	The class of a pencil in $\abs{L}$ is
	\[
		-5\eta^{5} + \eta^{4} \theta,
	\]
	so the intersection of $\iota^* \ch(\pi_*^{(6)} H)$ with a line has degree $160 n - 176$.
	The class of a point is obviously $\eta^{6}$, so the codimension $0$ term of $\iota^* \ch(\pi_*^{(6)} H)$ has degree $64$.
	Therefore
	\[
		\iota^* \ch(\pi_*^{(6)} H) = 64 + \left(160 n - 176\right)h + \left(160 n^{2} - 400 n + 244\right)h^{2},
	\]
	where $h \eq c_1(\cO_{\abs{L}}(1))$.
	Finally \cref{hilbert polynomial integral} gives
	\begin{align*}
		\chi(\cO_S(n \Tt))
		&= \int_{\abs{L}} \left(64 + \left(160 n - 176\right)h + \left(160 n^{2} - 400 n + 244\right)h^{2}\right) \left(1 + \frac{3}{2}h + h^{2}\right) \\
		&= \int_{\abs{L}} 64 + \left(160 n - 80\right)h + \left(160 n^{2} - 160 n + 44\right)h^{2} \\
		&= 160 n^{2} - 160 n + 44,
	\end{align*}
	as required.
\end{proof}

\begin{proposition}\label{identifyfiber}
	$V = \oW$.
\end{proposition}
\begin{proof}
	Since $\oW \subseteq V$, it suffices to show that the Hilbert polynomial of $\oW$ is also $20 n^2 - 40 n + 22$.
	The projection $\tB \to B$ factors through $P \tou\varphi B$.
	Moreover $\tB$ moves in a pencil of divisors on $P$ spanned by $B_0 + \varphi^{-1}(\Xi_\beta)$ and $B_\infty + \varphi^{-1}(\Xi)$ \cite[1.2]{IzadiWang2018}.
	Therefore $\tB$ is polarized by
	\[
		\res{(B_0 + \varphi^{-1}(\Xi_\beta))}{\tB} = \tXi + \tXi_\beta + \Delta.
	\]
	This restricts to a divisor on $\tW$, which, by \cref{blowupintersections}, can be identified with the divisor
	\begin{equation}\label{eqPhi}
\Phi \eq W_d + (W_d + \beta) + V_{d + e + f} = W_d + (\Xi_\beta \cap W)
	\end{equation}
	on $W$.
	Write $\xi$ for the restriction of $[\Xi]$ to $W$, considered as an algebraic equivalence class, and let $\pt$ be the class of a point.
	Note that $\xi^2 = 4! \pt = 24 \pt$.
	The normal bundle sequence for $W \into B$ gives
	\[
		\todd(W) = \todd(\cO_B(\Xi))^{-2} = \left(1 + \frac{\xi}{2} + 2 \pt\right)^{-2} = 1 - \xi + 14 \pt.
	\]
	By \cref{manyIntersections}\cref{algequivalent} and \cref{eqPhi}, $[\Phi] = \xi + \omega$, where $\omega \eq [W_d]$.
	Therefore
	\[
		\frac{1}{2} [\Phi]^2 = \frac{1}{2} (24 \pt + 2 \xi \omega) = 12 \pt + \xi \omega = 20 \pt
	\]
	(see \cref{intersection numbers}), which gives
	\[
		\ch(\cO_W(n \Phi)) \todd(W) = (1 + (\xi + \omega) n + 20 \pt n^2) (1 - \xi + 14 \pt).
	\]
	The coefficient of $\pt$ in this expression is
	\[
		\chi(\cO_W(n \Phi)) = 20 n^2 - (\xi + \omega) \xi n + 14 = 20 n^2 - 32 n + 14.
	\]
	The quotient map $W \tou\psi \oW$ identifies the disjoint curves $W_d$ and $W_d + \beta$.
	There is a short exact sequence
	\[
		0 \to \cO_\oW \to \psi_* \cO_W \to \psi_* \cO_{W_d} \to 0.
	\]
	Twisting by $\Psi \eq \res{\T_0}{\oW}$ and using the projection formula gives
	\[
		0 \to \cO_\oW(n \Psi) \to \psi_* \cO_W(n \Phi) \to \psi_* \cO_{W_d}(n \Phi) \to 0.
	\]
	Since $n \Phi$ has degree $8 n$ on the genus $9$ curve $W_d$, the Hilbert polynomial of $\oW$ is
	\[
		20 n^2 - 32 n + 14 - (8 n - 8) = 20 n^2 - 40 n + 22,
	\]
	as required.
\end{proof}

\begin{lemma}\label{specialPullbacks}
	If $\xi \eq [\Xi]$ and $\omega \eq [W_d]$, then
	\[
		[\tW] = (\xi^2, \xi - \omega) \text{ in } \achow_2(\tB) = \achow_2(B) \oplus \achow_1(\Xi \cap \Xi_\beta),
	\]
	where $\achow_d$ is the group of $d$-dimensional algebraic cycles modulo algebraic equivalence.
	Consequently $[\tW]^2 = 16$.
\end{lemma}
\begin{proof}
	Let $\varphi : \tB \to B$ be the blowup and $\psi : \Delta \to \Xi \cap \Xi_\beta$ its restriction to $\Delta$.
	By standard properties of blowing up \cite[0.1.3.ii]{Beauville1977-2}
	\[
		[\tW] = (\varphi_* [\tW], \psi_* (\res{[\tW]}{\Delta})) = ([W], [V_{d + e + f}]) = (\xi^2, \xi - \omega).
	\]
	Therefore, by \cref{intersection numbers redux},
	\[
		[\tW]^2 = \int_B \xi^4 + \int_\Delta c_1(\cO_\Delta(\Delta)) \psi^*(\xi - \omega)^2 = 4! - \int_{\Xi \cap \Xi_\beta} \xi^2 - 2 \xi \omega + \omega^2 = 2 \int_B \frac{\xi^4}{3} = 16,
	\]
	as required.
\end{proof}

\section{Proof of the main theorem}\label{secProof}

\begin{proposition}
	If $t \in \tT$ and $\T_t$ is smooth, then $[\cV_t]^2 = 16$. 
\end{proposition}
\begin{proof}
	The class $[\cV] \in \chow_3(\T_\tT)$ defines a family of $0$-cycle classes (i.e., a $1$-cycle) $[\cV]^2 \in \chow_1(\T_\tT)$ over $\tT$, and $[\cV_t]^2$ is the degree of the specialization of $[\cV]^2$ at $t$ \cite[10.1]{fulton84}.
	Since $\T_\tT \to \tT$ is flat, specialization at $t$ is the same as restricting to $\T_t$, for any $t \in \tT$.
	We can specialize $[\cV]^2$ to the central fiber, but since $\T_0$ is singular the meaning of $[V]^2$ is not clear.
	
	To remedy this, we pass to the operational Chow rings $\chow^*(\T_D)$ and $\chow^*(\T_\tT)$, which act on $\chow_*(\T_D)$ and $\chow_*(\T_\tT)$ via cap product.
	There is a unique ``Poincar\'e dual'' $\nu \in \chow^2(\T_\tT)$ such that $\nu \cap [\T_\tT] = [\cV]$ \cite[17.4]{fulton84}.
	Using the cap product on Chow groups \cite[8.1]{fulton84} for the inclusion $\iota : \T_0 = \T_D \into \T_\tT$, one checks that $\iota^* \nu \cap [\T_D] = [V]$ and $\iota^* \nu^2 \cap [\T_D] = \iota^* [\cV]^2$.
	It follows that
	\[
		\int_{\T_t} [\cV_t]^2 = \int_{\T_D} \iota^* [\cV]^2 = \int_{\T_D} \iota^* \nu \cap [V] = \int_{\T_D} \iota^* \nu \cap \psi_* [\tW] = \int_\tB \psi^* \iota^* \nu \cap [\tW] = \int_\tB (\iota \psi)^* [\cV] \cdot [\tW],
	\]
	where $\psi : \tB \to \T_0$ denotes the normalization.
	If $\delta \eq (\iota \psi)^* [\cV] - [\tW]$ then
	\[
		\psi_* \delta = \psi_*(\psi^* \iota^* \nu \cap [\tB]) - \psi_* [\tW] = \iota^* \nu \cap [\T_D] - [V] = 0.
	\]
	Let $U \eq \tB \minus (\tXi \amalg \tXi_\beta)$ be the smooth locus of $\T_0$, and consider the localization sequences
	\begin{tikzcenter}[node distance = 10ex and 7em]
		\node (ch1BU) {$\chow(U, 1)$};
		\node[right = of ch1BU] (chXi2) {$\chow(\tXi \amalg \tXi_\beta)$};
		\node[right = of chXi2] (chB) {$\chow(\tB)$};
		\node[right = of chB] (chBU) {$\chow(U)$};
		\node[right = of chBU] (chBx) {$0$};

		\node[below = of ch1BU] (ch1TU) {$\chow(U, 1)$};
		\node[right = of ch1TU] (chXi) {$\chow(\Xi)$};
		\node[right = of chXi] (chT) {$\chow(\T_0)$};
		\node[right = of chT] (chTU) {$\chow(U)$};
		\node[right = of chTU] (chTx) {$0$};

		\draw[-To]
			(ch1BU) edge (chXi2)
			(chXi2) edge (chB)
			(chB) edge (chBU)
			(chBU) edge (chBx)

			(ch1TU) edge (chXi)
			(chXi) edge (chT)
			(chT) edge (chTU)
			(chTU) edge (chTx)

			(ch1BU) edge[Eq] (ch1TU)
			(chXi2) edge node {$\psi_{\Xi *}$} (chXi)
			(chB) edge node {$\psi_*$} (chT)
			(chBU) edge[Eq] (chTU);
	\end{tikzcenter}
	(here $\chow(U, 1)$ is one of Bloch's higher Chow groups \cite{Bloch94}, but all we need is a group depending only on $U$, which is easy to construct with a little thought).
	Since $\psi_* \delta =0$, the diagram implies $\res{\delta}{U} =0$, hence one can find $\gamma \in \chow_2(\tXi \amalg \tXi_\beta)$ mapping to $\delta \in \chow_2(\tB)$.
	After possibly subtracting an element of $\chow(U, 1)$, we may assume that $\psi_{\Xi *} \gamma = 0 \in \chow_2(\Xi)$.
	\Cref{blowupintersections} implies that
	\[
		\int_\tB \delta \cdot [\tW] = \int_{\tXi \amalg \tXi_\beta} \gamma \cdot \res{[\tW]}{\tXi \amalg \tXi_\beta} = \int_\Xi \psi_{\Xi *} \gamma \cdot [W_d] = 0.
	\]
	Therefore $[\cV_t]^2 = [\tW]^2 = 16$, as required.
\end{proof}

\begin{proof}[Proof of \cref{thmmain}]
	For a very general ppav $(A, \T)$, one constructs a degeneration as above to get $[\tX_\lambda]^2 = 16$ for every Prym-curve $X$.
	The other intersection numbers are given by \cref{intersection numbers}.

	If $X$ and $Y$ are Prym-curves corresponding to lines $E$ and $F$ on a cubic surface, then \cref{intersection numbers} says that $[\tX_\lambda] . [\tY_\lambda] = 2 (7 - E . F)$.
	Choose mutually unrelated Prym-curves $X_1, \dots, X_6$ and another one $Y$ related to two of the $X_i$.
	The argument of \cref{Wgenerators} shows that the lattice $\bL$ spanned by classes of the form $[\tX_{i \lambda}] - [\tY_\lambda]$ is isometric to $\esix(2)$, and therefore spans $\bK_\bQ^+$.
	It is straightforward to check, e.g., by the nondegeneracy of the pairing, that $\bL$ contains $[\tX_\lambda] - [\tZ_\lambda]$ for any two Prym-curves $X$ and $Z$.
	The $\bQ$-vector space spanned by the $[\tX_\lambda]$ is at least $7$-dimensional, as $[\tX_\lambda] \notin \bK$, so it has to be $\bQ [\T]^2 + \bK_\bQ^+$.
\end{proof}

\section{Appendix}\label{secApp}

\subsection{Cohomological Calculations} Let $X$ be a smooth curve of genus $6$ and $\tX \tou\pi X$ an \'etale double cover.

\begin{lemma}\label{symmetric power etale pushforward}
	Let $\eta, \theta \in H^2(X^{(6)}, \bZ)$ be, respectively, the class of $X^{(5)}$ (plus a point), and the class of the polarization inherited from $\Pic^6(X)$.
	If $\teta, \ttheta \in H^2(\tX^{(6)}, \bZ)$ are the corresponding classes for $\tX$, then $\pi_*^{(6)} : H^*(\tX^{(6)}, \bZ) \to H^*(X^{(6)}, \bZ)$ has the following description:
	\[
		\begin{array}{|c|c|c|c|c|c|}
			\hline
			\alpha & \teta & \ttheta & \teta^2 & \teta \ttheta & \ttheta^2 \\
			\hline
			\pi_*^{(6)} \alpha & 32 \eta & 32 (5 \eta + \theta) & 16 \eta^2 & 16 \eta (5 \eta + \theta) & 16 (20 \eta^2 + 10 \eta \theta + \theta^2) \\
			\hline
		\end{array}
	\]
\end{lemma}
\begin{proof}
	Let $(\lambda_1, \mu_1, \dots, \lambda_6, \mu_6)$ be a symplectic basis for $H^1(X, \bZ)$ (with $\lambda_i \cdot \mu_i = 1$ for all $i$).
	Using the picture in \cite[Proposition 12.4.2]{BirkenhakeLange04} we can construct a symplectic basis
	\[
		\tlambda_1, \tmu_1, \lambda_2^+, \mu_2^+, \lambda_2^-, \mu_2^-, \dots, \lambda_6^+, \mu_6^+, \lambda_6^-, \mu_6^-
	\]
	for $H^1(\tX, \bZ)$ such that $\pi_* : H^1(\tX, \bZ) \to H^1(X, \bZ)$ sends $\tlambda_1 \mapsto 2 \lambda_1$ and $\tmu_1 \mapsto \mu_1,$ while the $\lambda_i^\pm$ and $\mu_i^\pm$ are sent to $\lambda_i$ and $\mu_i$ respectively.
	Let $\rho_k : X^6 \to X$ be the projections, and note that
	\[
		\sum_{k = 1}^6 \rho_k^* \lambda_i, \sum_{k = 1}^6 \rho_k^* \mu_i \text{ and } \sum_{k = 1}^6 \rho_k^*(\lambda_i \mu_i)
	\]
	all descend to classes in $H^*(X^{(6)}, \bZ)$.
	We denote the first two by $\xi_i$ and $\zeta_i$ respectively; the latter is independent of $i$ and descends to $\eta$ \cite[(3.1) and (14.2)]{Macdonald62}.
	Moreover $\theta = \sum_{i = 1}^6 \sigma_i$, where $\sigma_i \eq \xi_i \zeta_i$.
	Similarly $\ttheta$ is the sum of $\tsigma_1 \eq \txi_1 \tzeta_1$ and the $\sigma_i^\pm \eq \xi_i^\pm \zeta_i^\pm$.

	Since $H^*(X^{(6)}, \bZ)$ is torsion-free \cite[(12.3)]{Macdonald62}, while $H^*(X^{(6)}, \bZ) \into H^*(X^6, \bZ) \to H^*(X^{(6)}, \bZ)$ is multiplication by $6!$ (and likewise for $\tX$), we can compute $\pi_*^{(6)}$ using
	\[
		\pi_*^{\times 6} : H^*(\tX^6, \bZ) \to H^*(X^6, \bZ).
	\]
	Note that the cross product is natural for morphisms of even relative (real) dimension (up to a sign in general) \cite[5.3.10 and 5.6.21]{Spanier}.
	In particular $\pi_*^{(6)} \teta$ corresponds to
	\[
		\sum_{k = 1}^6 \pi_*^{\times 6} \trho_k^*(\tlambda_1 \tmu_1) = \sum_{k = 1}^6 \rho_1^*(\pi_* 1) \cdots \rho_k^*(\pi_*(\tlambda_1 \tmu_1)) \cdots \rho_6^*(\pi_* 1) = 32 \sum_{k = 1}^6 \rho_k^*(\lambda_1 \mu_1),
	\]
	where $\trho_k : \tX^6 \to \tX$ are the projections.
	Therefore $\pi_*^{(6)} \teta = 32 \eta$.
	Similarly $\pi_*^{(6)} \tsigma_1 = 32 \sigma_1$, because
	\begin{align*}
		\sum_{k = 1}^6 \sum_{l = 1}^6 \pi_*^{\times 6}(\trho_k^* \tlambda_1 \trho_l^* \tmu_1)
		&= \sum_{k = 1}^6 \left(\pi_*^{\times 6} \trho_k^*(\tlambda_1 \tmu_1) + \sum_{l \neq k} \pi_*^{\times 6}(\trho_k^* \tlambda_1 \trho_l^* \tmu_1)\right) \\
		&= \sum_{k = 1}^6 \left(32 \rho_k^*(\lambda_1 \mu_1) + 16 \sum_{l \neq k} \rho_k^*(2 \lambda_1) \rho_l^* \mu_1\right) \\
		&= 32 \sum_{k = 1}^6 \sum_{l = 1}^6 \rho_k^*\lambda_1 \rho_l^* \mu_1.
	\end{align*}
	The calculation for $\pi_*^{(6)} \sigma_i^\pm = 16 (\eta + \sigma_i)$ is almost the same, but the terms with $l = k$ and $l \neq k$ have different coefficients (since $2 \lambda_1$ becomes $\lambda_i$).
	The sum of the $l \neq k$ terms with $\frac{1}{2}$ times the $l = k$ terms descends to $16 \sigma_i$.
	We are left with
	\[
		16 \sum_{k = 1}^6 \rho_k^*(\lambda_i \mu_i),
	\]
	which descends to $16 \eta$.
	Adding these up for all $i$ gives the required formula for $\pi_*^{(6)} \ttheta$.

	The projection formula implies that $\pi^*(\lambda_1 \mu_1) = 2 \tlambda_1 \tmu_1$, so $\pi^{(6) *} \eta = 2 \teta$ by definition.
	Therefore
	\[
		\pi_*^{(6)}(2 \teta^2) = \pi_*^{(6)}(\teta \pi^{(6) *} \eta) = 32 \eta^2,
	\]
	giving the desired formula for $\pi_*^{(6)}(\teta^2)$.
	The calculation for $\pi_*^{(6)}(\teta \ttheta)$ is similar.
	
	Computing $\pi_*^{(6)}(\ttheta^2)$ is more complicated.
	If $i \neq j$ then $\sigma_i^\pm \sigma_j^\pm$ corresponds to
	\begin{equation}\label{sigma sigma}
		\sum_{k = 1}^6 \sum_{l = 1}^6 \sum_{m = 1}^6 \sum_{n = 1}^6 \trho_k^* \lambda_i^\pm \trho_l^* \mu_i^\pm \trho_m^* \lambda_j^\pm \trho_n^* \mu_j^\pm.
	\end{equation}
	Let $\tS$ be the sum of the terms in \eqref{sigma sigma} whose indices are distinct.
	Also let $\tT$, $\tU$ and $\tV$ the sums of those for which $k = l$, $m = n$, or both (with all the other pairs distinct).
	The remaining terms of \eqref{sigma sigma} all vanish (e.g.\ when $k = m$, such a term contains $\rho_k^*(\lambda_i^\pm \lambda_j^\pm) = 0$), so $\tS + \tT + \tU + \tV$ descends to $\sigma_i^\pm \sigma_j^\pm$.
	Moreover, for degree reasons
	\[
		\tV = \sum_{k = 1}^6 \sum_{m \neq k} \trho_k^*(\lambda_i^\pm \mu_i^\pm) \trho_m^*(\lambda_j^\pm \mu_j^\pm) = \sum_{k = 1}^6 \sum_{m = 1}^6 \trho_k^*(\lambda_i^\pm \mu_i^\pm) \trho_m^*(\lambda_j^\pm \mu_j^\pm)
	\]
	descends to $\teta^2$.
	Similarly, $\tT + \tV$ and $\tU + \tV$ correspond to $\teta \sigma_j^\pm$ and $\sigma_i^\pm \teta$ respectively.

	Define $S, T, U, V \in H^4(X^{(6)}, \bZ)$ in an analogous way.
	Since each term in $\tS$ involves four distinct projections (with two missing), $\pi_*^{\times 6} \tS = 4 S$.
	Applying this reasoning to the other sums gives
	\[
		\pi_*^{\times 6}(\tS + \tT + \tU + \tV) = 4 S + 8 (T + U) + 16 V = 4 ((S + T + U + V) + (T + V) + (U + V) + V),
	\]
	and hence
	\[
		\pi_*^{(6)}(\sigma_i^\pm \sigma_j^\pm) = 4 (\sigma_i \sigma_j + \eta \sigma_j + \sigma_i \eta + \eta^2) = 4 (\eta + \sigma_i) (\eta + \sigma_j).
	\]

	The same formula clearly works for $\sigma_i^\pm \sigma_j^\mp$.
	However, it breaks down for $\sigma_i^\pm \sigma_i^\mp$, because the corresponding class in $H^4(X^6, \bZ)$ has extra terms with $k = n$ and $l = m$ (so $S + T + U + V$ does not correspond to $\sigma_i^2$).
	In this case (i.e., when $i = j$), swapping $k$ and $m$ takes each term of $S$ to its negative, so $S = 0$.
	On the other hand, $T + V$ and $U + V$ both correspond to $\eta \sigma_i$.
	It follows that
	\[
		\pi_*^{(6)}(\sigma_i^\pm \sigma_i^\mp) = 8 \eta \sigma_i + 8 \eta \sigma_i = 16 \eta \sigma_i,
	\]
	because it corresponds to $8 (T + U) + 16 V = 8 (T + V) + 8 (U + V)$.

	The calculation for $\tsigma_1 \sigma_j^\pm$ is also a little different, because $S$ and $U$ pick up a factor of two when pushing forward $\tlambda_1$ ($T$ and $V$ do not because they involve $\tlambda_1 \tmu_1$ instead).
	In other words
	\[
		\pi_*^{\times 6}(\tS + \tT + \tU + \tV) = 8 (S + T) + 16 (U + V) = 8 ((S + T + U + V) + (U + V))
	\]
	and hence
	\[
		\pi_*^{(6)}(\tsigma_1 \sigma_j^\pm) = 8 (\sigma_1 \sigma_j + \sigma_1 \eta) = 8 \sigma_1 (\eta + \sigma_j).
	\]
	The formula for $\pi_*^{(6)}(\ttheta^2)$ can be found by adding all of these terms.
\end{proof}

\begin{remark}
	The general formula, for a curve $X$ of genus $g$, is
	\[
		\pi_*^{(d)}(\teta^p \ttheta^q) = 2^{d - p - q} \sum_{l = 0}^q \binom{g - 1}{q - l} \frac{q!}{l!} \eta^{p + q - l} \theta^l.
	\]
	This can be proved by generalizing the above argument (see \cite{Conder-2019} for details).
\end{remark}

\subsection{The ranks of $\bK^+$ and $\bK^-$}\label{Appranks}

Although the results of this Paragraph follow from \cref{Apphodgenums}, we feel that it is worth including because the proof here is considerably simpler than the general calculation of Hodge numbers in \cref{Apphodgenums} and also provides an independent verification of \cref{Apphodgenums}.

We know that $\Theta \into A$ satisfies the Lefschetz hyperplane theorem for
$H^*(-, \bQ)$. Since pullback and Gysin maps on cohomology are $(-1)$-equivariant,
the theorem also holds for $H^*(-, \bQ)^+$:

\begin{proposition}\label{propLefschetz}
\begin{enumerate} 
	\item For $n< g-1$, $H^n(\Theta,\bQ)^+ \gets H^n(A, \bQ)^+$ is an
		isomorphism. 
	\item For $g-1 < n \leq 2g-2$, $H^n(\Theta, \bQ)^+ \to H^{n+2}(A,
		\bQ)^+$ is an isomorphism.
	\item For $n=g-1$, the pullback $H^{g-1}(A,\bQ)^+ \to
		H^{g-1}(\Theta,\bQ)^+$ is an injection. \qed
\end{enumerate} 
\end{proposition}

Let $A[2]$ be the set of $2$-torsion points of $A$, and put $\Theta[2] := A[2] \cap \Theta$. We have $\#A[2] = 2^{2g}$ and, since $\T$ is smooth, $\#\T[2] = 2^{g-1}(2^g-1)$. 
Also let $A \tou\pi A^+$ and $\T \tou\pi \T^+$ be the quotients of $A$ and $\T$ by $-1$.

\begin{lemma}
	$\chi(A^+)=2^{2g-1}$ and $\chi(\Theta^+)=\frac{1}{2} (-1)^{g-1}g! + 2^{g-2}(2^g-1)$
\end{lemma}
\begin{proof}
It is well-known that $\chi(A)=0$ and $\chi(\Theta)=(-1)^{g-1}g!$. Since $A$ is a double cover of $A^+$ ramified at
$2^{2g}$ points, $\chi(A) = 2\chi(A^+) - 2^{2g}$ and hence $\chi(A^+) = 2^{2g-1}$.  Similarly
\begin{align*}
\chi(\Theta^+) &= \frac{1}{2} ( \chi(\Theta) + \#\Theta[2] ) = \frac{1}{2} (-1)^{g-1}g! + 2^{g-2}(2^g-1). \qedhere
\end{align*}
\end{proof}
\begin{proposition}
	We have
	\[\dim \frac{H^{g-1}(\Theta, \bQ)^+}{H^{g-1}(A, \bQ)^+} = \frac{g!}{2} + (-1)^g 2^{g-2}(2^g+1) + 
		\left\{\begin{array}{cc} 
			-\binom{2g}{g}    &g\,\,\mathrm{ even} \\
			\binom{2g}{g+1} &g\,\,\mathrm{ odd} \\
		\end{array} \right.
	\]
\end{proposition}
\begin{proof}
	By \cite[(1.2)]{Macdonald62} and the injectivity of Proposition \ref{propLefschetz},
	\[
		\dim \frac{H^{g-1}(\Theta, \bQ)^+}{H^{g-1}(A, \bQ)^+} = h^{g-1}(\Theta^+)-h^{g-1}(A^+).
	\]
	We have $2\left(\chi(A^+)-\chi(\T^+)\right)= \#A[2]-\#\Theta[2]+(-1)^g g! =
	2^{g-1}(2^g+1)+(-1)^g g!$. By Proposition \ref{propLefschetz}, we can match up all
	cohomology groups of $\T^+$ with isomorphic cohomology groups of
	$A^+$, with the exception of degree $g-1$ on the $\T^+$ side and
	degrees $g-1, g, g+1$ on the $A^+$ side.  Thus
	\[
		\chi(A^+)-\chi(\T^+)
		=(-1)^g\left( h^{g-1}(\T^+) - h^{g-1}(A^+) +
		h^g(A^+)-h^{g+1}(A^+) \right),
	\]
	which means
	\[
		h^{g-1}(\Theta^+)-h^{g-1}(A^+) = (-1)^g \frac{2^{g-1}(2^g+1) + (-1)^g g!}{2} + (-1)^g (h^{g + 1}(A^+) - h^g(A^+)).
	\]
	Finally note that $h^n(A^+) = 0$ for $n$ odd and $h^n(A^+) = \binom{2g}{n}$ for
	$n$ even.
\end{proof}

\begin{proposition}
	\[ \dim \bK^-_\bQ = \frac{g!}{2} + (-1)^{g-1}(2^{g-2}(2^g+1))+\left\{ \begin{array}{cc}
			\binom{2g}{g+1} &g\,\,\mathrm{even} \\
			-\binom{2g}{g}   &g\,\,\mathrm{odd}
	\end{array}\right.\]
\[ \dim \bK^+_\bQ = \frac{g!}{2} +(-1)^g(2^{g-2}(2^g+1))+
		\left\{\begin{array}{cc} 
			-\binom{2g}{g}    &g\,\,\mathrm{ even} \\
			\binom{2g}{g+1} &g\,\,\mathrm{ odd} \\
		\end{array} \right.
\]
\end{proposition}
\begin{proof}
Taking $-1$ invariants of the exact sequence
\[ 0 \to H^{g-1}(A, \bQ) \to H^{g-1}(\Theta, \bQ) \to \bK_\bQ \to 0\]
gives an isomorphism between the $-1$ invariant component of $\bK$ and
$\frac{H^{g-1}(\T, \bQ)^+}{H^{g-1}(A, \bQ)^+}$.  
On the other hand, $\dim \bK_\bQ = g! + \binom{2g}{g+1} - \binom{2g}{g}$.  After
subtracting and using the previous formula we get the statement.
\end{proof} 

\subsection{The Hodge numbers of $\bK^+$ and $\bK^-$}\label{Apphodgenums}

\begin{proposition}\label{primal Hodge numbers}
	The Hodge numbers are
	\[
		h^{p, g - 1 - p}(\bK) = \eulerian{g}{p} - \binom{g}{p} \binom{g - 1}{p} + \binom{g}{p + 1} \binom{g - 1}{p - 1},
	\]
	where $\eulerian{g}{p} \eq \sum_{k = 0}^p \binom{g + 1}{k} (-1)^k (p + 1 - k)^g$ is an eulerian number (see \cite[(6.38)]{GrahamKnuthPatashnik94}).
\end{proposition}
\begin{proof}
	The Hodge number $h^{p, g - 1 - p}(\bK)$ appears in
	\begin{align*}
		\chi(\Omega_\T^p)
		&= \sum_{k = 0}^{g - 1} (-1)^k h^{p, k}(\T) \\
		&= \sum_{k = 0}^{g - 1 - p} (-1)^k h^{p, k}(A) + (-1)^{g - 1 - p} h^{p, g - 1 - p}(\bK) + \sum_{k = g - p}^{g - 1} (-1)^k h^{p + 1, k + 1}(A).
	\end{align*}
	Using a standard identity \cite[(5.16)]{GrahamKnuthPatashnik94}, the first term simplifies to
	\[
		\sum_{k = 0}^{g - 1 - p} (-1)^k \binom{g}{p} \binom{g}{k} = (-1)^{g - 1 - p} \binom{g}{p} \binom{g - 1}{p},
	\]
	and similarly $\sum_{k = g - p}^{g - 1} (-1)^k h^{p + 1, k + 1}(A) = (-1)^{g - p} \binom{g}{p + 1} \binom{g - 1}{p - 1}$.
	Next, we will compute $\chi(\Omega_\T^p)$.
	The conormal bundle sequence for $\T \subset A$ induces exact sequences
	\[
		0 \to \Omega_\T^{p - 1}(-\T) \to \res{\Omega_A^p}{\T} \to \Omega_\T^p \to 0
	\]
	for all $p \in \set{1, \dots, g}$, so by induction (and the triviality of $\Omega_A$)
	\[
		\ch(\Omega_\T^p) = \sum_{k = 0}^p \binom{g}{k} \res{(-e^{-\theta})^{p - k}}{\T}.
	\]
	Moreover $\todd(\T) = \todd(\cO_\T(\T))^{-1} = \res{\frac{1 - e^{-\theta}}{\theta}}{\T}$, and hence
	\begin{align*}
		\chi(\Omega_\T^p)
		&= \int_\T \ch(\Omega_\T^p) \todd(\T) \\
		&= \int_A \sum_{k = 0}^p \binom{g}{k} (-e^{-\theta})^{p - k} (1 - e^{-\theta}) \\
		&= \int_A \left(\sum_{k = 0}^p \binom{g + 1}{k} (-e^{-\theta})^{p + 1 - k} + \binom{g}{p}\right) \\
		&= \sum_{k = 0}^p \binom{g + 1}{k} (-1)^{p + 1 - k} (k - p - 1)^g,
	\end{align*}
	since $\int_A \frac{\theta^g}{g!} = \chi(\cO_A(\T)) = 1$.
	After some rearranging, this gives the required formula.
\end{proof}

A similar argument can be used to compute the Hodge numbers of $\bK^+$, as follows.
Let $\Tt \tou\beta \T$ be the blowup of $\T$ at its $2$-torsion points, with $\Delta \subset \Tt$ the exceptional divisor.
If $\Tt \tou\pi \Tt^+$ is the quotient of $\Tt$ by the induced action of $-1$, then $\pi^* : H^*(\Tt^+, \bQ) \to H^*(\Tt, \bQ)^+$ is an isomorphism \cite[(1.2)]{Macdonald62}.
Set $\cO_\Tt(1) \eq \cO_\Tt(-\Delta)$.
For each $p \in \bN$ there are exact sequences
\[
	0 \to (\beta^* \Omega_\T^p)(1 - p) \to \Omega_\Tt^p \to \Omega_\Delta^p \to 0
\]
and
\[
	0 \to \pi^* \Omega_{\Tt^+}^{p + 1} \to \Omega_\Tt^{p + 1} \to \Omega_\Delta^p(1) \to 0.
\]
It is straightforward to verify this by local calculations (for coordinate-free proofs, see \cite[{\S}1.2]{Conder-2019}).
The following ``dual'' sequences will also be useful:
\[
	0 \to T_\Tt \to \beta^* T_\T \to T_\Delta(-1) \to 0
\]
and
\[
	0 \to T_\Tt \to \pi^* T_{\Tt^+} \to \cO_\Delta(-2) \to 0.
\]

\begin{proposition}
	The Hodge numbers $h^{p, g - 1 - p}(\bK^+)$ of $\bK^+$ are
	\[
		\frac{1}{2} \eulerian{g}{p} + (-1)^g \left(\binom{g}{p} \sum_{q = 0}^{g - 1 - p} \binom{g}{q} \epsilon_{p + q} + \binom{g}{p + 1} \sum_{q = g - p}^{g - 1} \binom{g}{q + 1} \epsilon_{p + q} - \binom{g - 1}{p} \frac{2^g - 1}{2}\right),
	\]
	where $\epsilon_k \eq \frac{1}{2} (1 + (-1)^k)$ is one (resp.\ zero) if $k$ is even (resp.\ odd).
\end{proposition}
\begin{proof}
	If $p = 0$ then $h^{p, g - 1 - p}(\bK^+) = h^{p, g - 1 - p}(\bK) = 0$.
	Otherwise, by Hirzebruch-Riemann-Roch and the above sequences
	\begin{align*}
		2 \chi(\Omega_{\Tt^+}^p)
		&= 2 \int_{\Tt^+} \ch(\Omega_{\Tt^+}^p) \todd(\Tt^+) \\
		&= \int_\Tt \pi^*(\ch(\Omega_{\Tt^+}^p) \todd(\Tt^+)) \\
		&= \int_\Tt (\ch(\Omega_\Tt^p) - \ch(\Omega_\Delta^{p - 1}(1))) \todd(\Tt) \todd(\cO_\Delta(-2)) \\
		&= \int_\Tt (\ch(\beta^* \Omega_\T^p) \ch(\cO_\Tt(1 - p)) + \ch(\Omega_\Delta^p) - \ch(\Omega_\Delta^{p - 1}(1))) \todd(\Tt) \todd(\cO_\Delta(-2)).
	\end{align*}
	The first term can be computed on $\T$.
	To be specific:
	\begin{align*}
		\chi_1
		&\eq \int_\Tt \ch(\beta^* \Omega_\T^p) \ch(\cO_\Tt(1 - p)) \todd(\Tt) \todd(\cO_\Delta(-2)) \\
		&= \int_\T \ch(\Omega_\T^p) \todd(\T) \cdot \beta_*\left(\ch(\cO_\Tt(1 - p)) \frac{\todd(\cO_\Delta(-2))}{\todd(T_\Delta(-1))}\right)
	\end{align*}
	The Euler sequence and the ideal sheaf sequence for $\Delta \subset \Tt$ imply that
	\[
		\todd(T_\Delta(-1)) = \frac{\todd(\cO_\Delta)^{g - 1}}{\todd(\cO_\Delta(-1))} = \frac{\todd(\cO_\Tt(1))^{1 - g}}{\todd(\cO_\Tt(-1))},
	\]
	and similarly $\todd(\cO_\Delta(-2)) = \todd(\cO_\Tt(-2)) \todd(\cO_\Tt(-1))^{-1}$, so
	\[
		\chi_1 = \int_\T \ch(\Omega_\T^p) \todd(\T) \cdot \beta_*(\ch(\cO_\Tt(1 - p)) \todd(\cO_\Tt(1))^{g - 1} \todd(\cO_\Tt(-2))).
	\]
	Everything inside $\beta_*$ is a polynomial in $h \eq c_1(\cO_\Tt(1))$.
	Since $\beta$ contracts $\Delta$ the only powers of $h$ which survive are $h^0$ and $h^{g - 1}$.
	Therefore
	\[
		\chi_1 = \chi(\Omega_\T^p) - \binom{g - 1}{p} \int_\Tt e^{(1 - p) h} \left(\frac{h}{1 - e^{-h}}\right)^{g - 1} \frac{2 h}{1 - e^{2 h}}.
	\]
	Since $-h$ is the class of $\Delta \subset \Tt$, which has $2^{g - 1} (2^g - 1)$ components,
	\[
		\chi_1 = \chi(\Omega_\T^p) - \binom{g - 1}{p} 2^{g - 1} (2^g - 1) \chi_2,
	\]
	where $\chi_2$ is the residue of
	\[
		-\frac{e^{(1 - p) z}}{(1 - e^{-z})^{g - 1}} \cdot \frac{2}{1 - e^{2 z}} = -\frac{e^{(g - p) z}}{(e^z - 1)^{g - 1}} \cdot \frac{2}{(1 - e^z) (1 + e^z)} = \frac{2 e^{(g - p) z}}{(e^z - 1)^g (e^z + 1)}
	\]
	at 0, or equivalently (making the change of variables $w \eq e^z - 1$), the residue of
	\begin{equation}\label{rational laurent}
		\frac{2 (w + 1)^{g - 1 - p}}{w^g (w + 2)}
	\end{equation}
	at $0$.
	Since $\frac{2}{w + 2} = \sum_{k = 0}^\infty \left(-\frac{w}{2}\right)^k$, taking the Laurent expansion of \cref{rational laurent} gives
	\[
		\chi_2 = \sum_{k = 0}^{g - 1 - p} \binom{g - 1 - p}{i} \left(-\frac{1}{2}\right)^{g - 1 - i} = \left(-\frac{1}{2}\right)^p \left(1 - \frac{1}{2}\right)^{g - 1 - p} = \frac{(-1)^p}{2^{g - 1}}.
	\]
	Grothendieck-Riemann-Roch allows us to compute the remaining terms on $\Delta$:
	\begin{align*}
		\chi_3
		&\eq \int_\Tt (\ch(\Omega_\Delta^p) - \ch(\Omega_\Delta^{p - 1}(1))) \todd(\Tt) \frac{\todd(\cO_\Tt(-2))}{\todd(\cO_\Tt(-1))} \\
		&= \int_\Delta (\ch(\Omega_\Delta^p) - \ch(\Omega_\Delta^{p - 1}(1))) \todd(\Delta) \frac{\todd(\cO_\Delta(-2))}{\todd(\cO_\Delta(-1))}.
	\end{align*}
	The following exact sequence arises from the Euler sequence on $\Delta$:
	\[
		0 \to \Omega_\Delta^p \to \cO_\Delta(-p)^{\oplus \binom{g - 1}{p}} \to \Omega_\Delta^{p - 1} \to 0.
	\]
	It implies that $\ch(\Omega_\Delta^p) = \binom{g - 1}{p} \ch(\cO_\Delta(-p)) - \ch(\Omega_\Delta^{p - 1})$, so by induction
	\[
		\ch(\Omega_\Delta^p) = \sum_{k = 0}^p (-1)^{p - k} \binom{g - 1}{k} e^{-k h}, 
	\]
	where $h \eq c_1(\cO_\Delta(1))$.
	Since
	\[
		\frac{\todd(\cO_\Delta(-2))}{\todd(\cO_\Delta(-1))} = \frac{2 h}{1 - e^{2 h}} \cdot \frac{1 - e^h}{h} = \frac{2}{e^h + 1},
	\]
	it follows that $\chi_3 = 2^{g - 1} (2^g - 1) \chi_4$, where $\chi_4$ is the residue of
	\begin{align*}
		&\aligneq \left(\sum_{k = 0}^p (-1)^{p - k} \binom{g - 1}{k} e^{-k z} + \sum_{k = 0}^{p - 1} (-1)^{p - k} \binom{g - 1}{k} e^{(1 - k) z}\right) \frac{2}{(1 - e^{-z})^{g - 1} (e^z + 1)} \\
		&= \left(\sum_{k = 0}^p (-1)^{p - k} \binom{g - 1}{k} e^{(g - 1 - k) z} + \sum_{k = 0}^{p - 1} (-1)^{p - k} \binom{g - 1}{k} e^{(g - k) z}\right) \frac{2}{(e^z - 1)^{g - 1} (e^z + 1)}
	\end{align*}
	at $0$.
	The above residue calculations, with some minor adjustments for the extremal terms in the sums, can be used to show that
	\[
		\chi_4 = \sum_{k = 0}^p \frac{(-1)^p}{2^{g - 2}} \binom{g - 1}{k} - \sum_{k = 0}^{p - 1} \frac{(-1)^p}{2^{g - 2}} \binom{g - 1}{k} + 2 (-1)^p \delta_p = \frac{(-1)^p}{2^{g - 2}} \binom{g - 1}{p} + 2 (-1)^p \delta_p,
	\]
	where $\delta_0 \eq 0$, $\delta_{g - 1} = 0$ and $\delta_p \eq 1$ for $0 < p < g - 1$.
	Therefore
	\begin{align}\label{twice chi theta quotient}
		2 \chi(\Omega_{\Tt^+}^p)
		\nonumber
		&= \chi_1 + \chi_3 \\
		\nonumber
		&= \chi(\Omega_\T^p) - (-1)^p \binom{g - 1}{p} (2^g - 1) + (-1)^p \binom{g - 1}{p} 2 (2^g - 1) + (-1)^p 2^g (2^g - 1) \delta_p \\
		&= \chi(\Omega_\T^p) + (-1)^p (2^g - 1) \left(\binom{g - 1}{p} + 2^g \delta_p\right).
	\end{align}
	Decomposing $H^*(\Tt, \bQ)$ as in \cite[0.1.3.ii]{Beauville1977-2} gives
	\begin{align}\label{chi theta quotient}
		\chi(\Omega_{\Tt^+}^p)
		\nonumber
		&= \sum_{q = 0}^{g - 1} (-1)^q \dim(H^{p, q}(\Tt, \bQ)^+) \\
		&= \sum_{q = 0}^{g - 1} (-1)^q \dim(H^{p, q}(\T, \bQ)^+) + (-1)^p 2^{g - 1} (2^g - 1) \delta_p.
	\end{align}
	Combining \cref{twice chi theta quotient,chi theta quotient} allows us to express $(-1)^{g - 1 - p} h^{p, g - 1 - p}(\bK^+)$ as
	\[
		\frac{1}{2} \chi(\Omega_\T^p) + (-1)^p \frac{2^g - 1}{2} \binom{g - 1}{p} - \sum_{q = 0}^{g - 1 - p} (-1)^q h^{p, q}(A) \epsilon_{p + q} - \sum_{q = g - p}^{g - 1} (-1)^q h^{p + 1, q + 1}(A) \epsilon_{p + q},
	\]
	as required.
\end{proof}


\providecommand{\bysame}{\leavevmode\hbox to3em{\hrulefill}\thinspace}
\providecommand{\MR}{\relax\ifhmode\unskip\space\fi MR }
\providecommand{\MRhref}[2]{%
  \href{http://www.ams.org/mathscinet-getitem?mr=#1}{#2}
}
\providecommand{\href}[2]{#2}

\end{document}